\documentclass[11pt,reqno]{amsart}
\usepackage{tikz}
\usepackage{subfig}
\usepackage{epstopdf}
\usepackage{epsfig}
\usepackage[utf8]{inputenc}
\usepackage[a4paper, total={6in, 8in}]{geometry}
\usepackage{amsfonts}
\usepackage{mathtools}
\usepackage{mathrsfs}
\usepackage{algorithm}
\usepackage{algorithmic}
\usepackage{appendix}
\usepackage{graphicx}
\usepackage{multirow}
\usepackage{accents}
\usepackage{xcolor}
\usepackage{empheq}
\allowdisplaybreaks

\graphicspath{ {./images/} }

\usepackage{tikz}
\usetikzlibrary{shapes,arrows}

\tikzstyle{decision} = [diamond, draw, fill=blue!20, 
    text width=4.5em, text badly centered, node distance=3cm, inner sep=0pt]
\tikzstyle{block} = [rectangle, draw, fill=blue!20, 
    text width=5em, text centered, rounded corners, minimum height=4em]
\tikzstyle{line} = [draw, -latex']
\tikzstyle{cloud} = [draw, ellipse,fill=red!20, node distance=3cm,
    minimum height=2em]

\usetikzlibrary{positioning}
\tikzset{main node/.style={circle,fill=blue!20,draw,minimum size=1cm,inner sep=0pt},  }

\numberwithin{equation}{section}
\theoremstyle{plain}

\newtheorem{theorem}{Theorem}

\newtheorem{remark}[theorem]{Remark}

\newcommand{\vx}{\mathbf{x}}

\newcommand{\vv}{\mathbf{v}}
\newcommand{\vq}{\mathbf{q}}
\newcommand{\vw}{\mathbf{w}}
\newcommand{\vm}{\mathbf{m}}

\newtheorem{example}{Example}
\newtheorem{model}{Model}
\newtheorem{test}{Test}

\author[Ding, Li, Osher, Yin]{Lisang Ding \and Wuchen Li \and Stanley Osher \and Wotao Yin}
\email{lisangding@ucla.edu}
\email{wcli@math.ucla.edu}
\email{sjo@math.ucla.edu}
\email{wotaoyin@math.ucla.edu}

\title[A mean field game inverse problem]{A mean field game inverse problem}

\begin{document}

\begin{abstract}
	Mean-field games arise in various fields including economics,  engineering and machine learning. They study strategic decision making in large populations where the individuals interact via certain mean-field quantities. The ground metrics and running costs of the games are of essential importance but are often unknown or only partially known. In this paper, we propose mean-field game inverse-problem models to reconstruct the ground metrics and interaction kernels in the running costs. The observations are the macro motions, to be specific, the density distribution and velocity field of the agents. They can be corrupted by noise to some extent. Our models are PDE constrained optimization problems, which are solvable by first-order primal-dual methods. Besides, we apply Bregman iterations to find the optimal model parameters. 		We numerically demonstrate that our model is both efficient and robust to noise.
\end{abstract}

\thanks{This paper is under the support of AFOSR MURI FA9550-18-1-0502.}
\maketitle


\section{Introduction}
Mean-field games (MFGs) study strategic decision making in large populations
where the individuals interact via certain mean-field quantities~\cite{lasry_jeux_2006,lasry_jeux_2006-1,cardaliaguet_master_2015,lasry_mean_2007}. In an MFG, the decision of each player depends on their state and interactions with, not just individual, but all other players. MFGs are used to study the strategies of players at a macro level. In recent years, MFGs have gained enormous popularity,
starting to play vital roles in many research fields including
economics~\cite{achdou_partial_2014,morel_mean_2011,achdou_income_2017,gomes_economic_2015}, finance~\cite{firoozi_optimal_2017,cardaliaguet_mean_2018,casgrain_mean_2019,lehalle_mean_2019,jaimungal_mean-field_2015}, engineering~\cite{de_paola_mean_2019,kizilkale_integral_2019,gomes_mean-field_2020} and machine learning~\cite{e_mean-field_2019,yang_mean_2018}.

One of the typical MFG formulations is the following transport-related optimization problem ~\cite{cardaliaguet_master_2015,lasry_mean_2007}:
\begin{align*}
\underset{\mbox{density,velocity}}{\mbox{minimize}}\quad & (\mbox{kinectic energy}) + (\mbox{regularization}) + (\mbox{final-time cost})\\
\mbox{subject to}\quad& \mbox{(transport equation)}\\
& \mbox{(initial density),}
\end{align*}
or written in math notation,
\begin{equation}
\label{MFG-Benemou}
\begin{split}
\underset{\rho,\vv}{\mbox{minimize\quad}}
&
\displaystyle{
	\int_0^T\left[\int_{\mathbb{T}^d} \frac{1}{2}\rho\vv^TG_M\vv dx+\mathcal{F}(\rho(\cdot,t))\right]dt+\mathcal{G}(\rho(\cdot,T))
}\\
\mbox{subject to\quad}&
\rho_t+\nabla \cdot (\rho\vv)=0\\
&\rho(\cdot,0)=\rho_0,
\end{split}
\end{equation}
where the problem is spatially defined in the $d$-dimensional torus $\mathbb{T}^d:=\mathbb{R}^d/\mathbb{Z}^d$, $\rho:\mathbb{T}^d\times [0,T]\rightarrow \mathbb{R}$ is the density distribution with its initial state set to $\rho_0$, $\vv\in\mathbb{R}^d$ is the velocity field, $G_M(x)$ is a $d\times d$ symmetric positive definite matrix (called the ground metric) that determines the kinetic energy consumed in different directions, 
$\mathcal{F}:\mathcal{P}(\mathbb{T}^d)\rightarrow \mathbb{R}$ is a convex functional that regularizes $\rho$, and finally $\mathcal{G}:\mathcal{P}(\mathbb{T}^d)\rightarrow \mathbb{R}$ is a convex functional of the terminal density distribution.

The objective in \eqref{MFG-Benemou} is the sum of kinetic energy and potential regularization. In practice, $\mathcal{F}(\cdot)$ can be quadratic function or the convolution function as in maximum mean discrepancy (MMD)~\cite{scholkopf_kernel_2007}, and $\mathcal{G}(\cdot)$ can be indicator function in typical optimal transport problem, or the distance between $\rho_T$ and its projection onto some convex set in~\cite{chow_algorithm_2019}. The continuity equation depicts that, the change of density mass equals to the flow-in mass minus the flow-out mass. In other words, the density of players can be viewed as compressible fluid.

In general, the MFG problem~(\ref{MFG-Benemou}) has no closed-form solution. Existence and uniqueness of a solution $(\rho,\vv)$ have been studied under suitable conditions in~\cite{cardaliaguet_master_2015,lasry_mean_2007,achdou_mean_2010}. There has been great progress in numerically solving the problem on a grid~\cite{achdou_mean_2010,bellomo_variational_2017,chow_algorithm_2017,chow_algorithm_2018,chow_algorithm_2019} or in a neural network~\cite{lin_apac-net:_2020}.

We call (\ref{MFG-Benemou}) the forward problem and name the problem of recovering the ground metric $G_M$ and the cost functionals $\mathcal{F}$ from the observations of $\rho$ and $\vv$  \emph{the inverse problem.}
The ground metric $G_M$ can depict the geometric structure of the sample space and decides the kinetic energy.
The cost functional $\mathcal{F}$, when taking a convolution form with a kernel $K$, can depict the total interaction energy among the players. When $G_M$ and $\mathcal{F}$ are unknown or partially known, learning them from the observable data becomes an important inverse problem.

In this paper, we recreate ground metric $G_M$ and convolution kernel $K$ from either clean or noisy observations of density distribution $\rho$ and velocity field $\vv$. We study this problem since $\rho$ can be observed directly or indirectly over time, and $\vv$ can also be measured directly from the game players. There are two scenarios in which we consider observation noise:
\begin{itemize}
	\item Players do not completely play to their optimal strategy, which can be different from the Nash equilibrium (NE).
	\item We apply an MFG (with infinite players) to approximate decision-making by a finite but large number players, causing a difference between the model and reality.
\end{itemize}


We take three steps to derive an inverse model in this paper. First, we deduce the KKT condition, which is an MFG PDE system whose solution equals to the minimizer of the optimization problem (\ref{MFG-Benemou}) under suitable conditions. Next, we add appropriate regularization (quadratic regularization on $\rho,\vv$, and $H^p$ norm on the reconstruction target) in the objective function. Combining above steps, we derive
the following inverse-problem model:
\begin{equation}
\label{intro-model}
\begin{split}
\underset{\theta,\rho,\vv}{\mbox{min}}\quad
&
\displaystyle{
	\frac{\alpha}{2}\|\rho-\hat{\rho}\|^2+\frac{\beta}{2}\|\vv-\hat{\vv}\|^2
	+\frac{\alpha_0}{2}
	\left(
	\|\rho_0-\hat{\rho}_0\|^2+\|\rho_T-\hat{\rho}_T\|^2
	\right)
	+\frac{\gamma}{p}\|\nabla \theta\|_p^p
}\\
\mbox{s.t.\quad}
&\rho_t+\nabla\cdot(\rho \vv)=0\\
&(G_M\vv)_t+\nabla\left(\frac{1}{2}\vv^T G_M\vv-\frac{\delta}{\delta\rho}\mathcal{F}(\rho)\right)=0\\
&\frac{\partial (G_M\vv)_i}{\partial{x_j}}=\frac{\partial (G_M\vv)_j}{\partial{x_i}}, \qquad i\not=j\\
&\int_{s_i(\hat{x}^i,\cdot)} (G_M\vv)_i dS_x=0,\quad\hat{x}^i\in\mathbb{T}^{d-1}, i=1,2,\ldots,d,
\end{split}
\end{equation}
where $(G_M\vv)_i$ and $(G_M\vv)_j$ denote the $i^{th}$ and $j^{th}$ component of vector $G_M\vv$, respectively. In this model, $\theta$ is the parameter, which is either the ground metric kernel $g_0$ determining $G_M$, or the convolution kernel linear in $\mathcal{F}$ determining the running cost. In objective function, $\hat{\rho},\hat{\vv}$ are observations, possibly noisy,
and the $L^2$ distance is applied for regularization. 
Then we regularize $\theta$ with $H^p$ norm, where $p=1,2,\dots$ can be selected according to the property of $\theta$. Here the first equation is the continuity equation, and the second one corresponds to a reformulation of Hamilton-Jacobi equation (HJE). We notice that the inverse model is non-convex due to the bi-linear constraints though the objective functional is convex. 


We numerically solve the inverse problem by discretizing (\ref{MFG-Benemou}) on a grid. 
The discrete inverse model is solved by a primal-dual method on its Lagrangian. In each iteration, we perform gradient descent to the primal variable and then update the dual variable with gradient ascent. The algorithm converges to a saddle point of the Lagrangian, which is a stationary point of the inverse model.

\textbf{Related work}: There are various approaches that successfully compute MFG (forward) problems with applications. Based on augmented Lagrangian, ~\cite{bellomo_variational_2017} solves the MFG via a primal-dual approach. ~\cite{li_parallel_2018} presents a parallel PDHG algorithm to compute the earth mover's distance, which is a special MFG type problem.  ~\cite{chow_algorithm_2017,chow_algorithm_2018,chow_algorithm_2019} introduce fast algorithms for the HJEs arising from optimal transport and MFG. ~\cite{liu_multilevel_2019,elamvazhuthi_optimal_2020} propose fast algorithms for Wasserstein-$p$ distances. For applications,~\cite{engquist_seismic_2018} applies optimal transport to seismic imaging, to be specific, selects the Wasserstein metric as a misfit function for full-waveform inversion. In addition, \cite{kachroo_inverse_2016} presents an inverse problem learning the traffic dynamics model via an MFG approach.

Recently, inverse problems for optimal transport have been studied. For example, \cite{li_learning_2018} proposes a unified data-driven framework to learn the adaptive, nonlinear interaction cost functions in the matching process from data corrupted by noise and then make predictions to new matchings. ~\cite{stuart_inverse_2020} proposes a framework to learn the unknown ground costs from noisy observations during optimal transport.
In particular, \cite{li_learning_2018,stuart_inverse_2020} focus on the linear programming formulation of inverse optimal transport problems.
Compared to existing works, we focus on PDE formulations of MFGs. Here the MFG system describes the dynamics of agents, where
we have observations about the motion and strategy adopted by the agents during the game.
Our observation is time dependent, which is different from the static joint distribution in \cite{li_learning_2018,stuart_inverse_2020}. In addition, for MFG with interaction energy, MFG dynamics can not be formulated as a minimizer of linear programming.


Here we summarize our contributions for MFG inverse problems as follows:
\begin{itemize}
	\item[(i)] We propose an inverse model for MFG. From the observation of feasible physical quantities, we recreate both ground metric and the interaction kernel function.
	\item[(ii)] We give a discrete format of the inverse problem on the grid.
	\item[(iii)] We provide a computational method for solving inverse MFG problems in an efficient fashion. Our approach is quite robust to noise in the observations.
	\item[(iv)] We apply Bregman iteration methods for the proposed constrained optimization.
\end{itemize}

\textbf{Organization}: The rest of the paper is organized as follows. In Section \ref{chap:Theory}, we briefly review some MFG theoretic basics and present two MFG models. In Section \ref{chap:MFG-model}, we deduce PDE systems equivalent to the MFG models and present the proposed inverse models, one for the ground metric $G_M$, and the other for the convolution kernel $K$. In Section \ref{chap:disc}, we discretize the MFG optimization problem (\ref{MFG-Benemou}) on a grid and do the same to their inverse models. Then, we apply primal-dual algorithms to solve the inverse models. {A Bregman approach to improve the algorithm performance is also proposed.} In Section \ref{chap:comp_result}, we present our computational results for the inverse models in both 1 and 2 dimensions. The presented results correspond to observations that are corrupted by noise at different levels.

\section{Review of mean-field games}
\label{chap:Theory}
In this section, we review the theoretical basics of MFG and give two special examples: regularized optimal transport and MMD interaction.

The standard potential MFG is already given in (\ref{MFG-Benemou}), where it has unknowns $\rho,\vv$. The objective function is the sum of kinetic and potential energies, with an additional terminal cost. The constraint is a continuity equation, depicting the macro motion of infinitely many players that are approximated by compressible fluid dynamics. For simplicity, in this paper, we assume the density distribution $\rho$ to be strictly positive,
\begin{equation*}
\begin{split}
\rho&\in\mathcal{P}_+(\mathbb{T}^d\times[0,T])=\Big\{
\rho(x,t)\in C^{1}(\mathbb{T}^d\times[0,T])	\mid \rho(x,t)>0,\int_0^T\int_{\mathbb{T}^d}\rho(x,t)dxdt<+\infty
\Big\},
\end{split}
\end{equation*}
also assume the terminal cost $\mathcal{G}$ in (\ref{MFG-Benemou}) is the indicator function of $\{\rho_T\}$, and the final-time state is given by $\rho_T$. 

Next,
we introduce the KKT condition of (\ref{MFG-Benemou}). Let $\vm=\rho\vv$, the product of density distribution and the velocity, denote the flux. Substitute $\vv$ with $\vm/\rho$ in (\ref{MFG-Benemou}); then, the optimization problem is transferred to a convex problem. Take $\varphi$ as the Lagrangian multiplier for the continuity equation. The Lagrangian of (\ref{MFG-Benemou}) is the sum of objective function and multiplied constraint violation:
\begin{equation*}
\begin{split}
&\operatorname*{min}_{\vm,\rho}\operatorname*{max}_{\varphi}
\int_{0}^{T}\int_{\mathbb{T}^d}\frac{1}{2}\frac{\vm^TG_M\vm}{\rho}
+
\varphi(\nabla\cdot  \vm+\rho_t)dxdt+\int_0^T\mathcal{F}(\rho(\cdot,t))dt\\
=
&\operatorname*{min}_{\vm,\rho}\operatorname*{max}_{\varphi}
\int_{0}^{T}\!\int_{\mathbb{T}^d}
\frac{1}{2}\frac{\vm^TG_M\vm}{\rho}
-\nabla\varphi^T\vm-\varphi_t\rho dxdt+\!\int_0^T\!\mathcal{F}(\rho(\cdot,t))dt
+\!\int_{\mathbb{T}^d}\!\varphi\rho dx\bigg|_{t=0}^{t=T}.
\end{split}
\end{equation*}

The optimization solution corresponds to a saddle point of the Lagrangian in density space. Since the optimization problem is convex, $\rho\in\mathcal{P}_+(\mathbb{T}^d)$ is strictly positive, the saddle point is exactly the point where the variation w.r.t. primal/dual variables vanishes.
Take the $L^2$ variation to the Lagrangian w.r.t. $\vm,\rho,\varphi$ over $\mathbb{T}^d\times (0,T)$. Then the KKT condition of (\ref{MFG-Benemou}) can be written as:
\begin{subequations}
	\label{MFG-KKT}
	\begin{align}
	\label{MFG-KKT1}
	\frac{G_M\vm}{\rho}-\nabla\varphi&=0\\
	\label{MFG-KKT2}
	-\frac{\vm^TG_M\vm}{2\rho^2}+\frac{\delta}{\delta\rho}\mathcal{F}(\rho)-\varphi_t&=0\\
	\label{MFG-KKT3}
	\nabla\cdot  \vm+\rho_t&=0,
	\end{align}
\end{subequations}
where $\delta$ in (\ref{MFG-KKT2}) is the $L^2$ variation.
From (\ref{MFG-KKT1}), we solve $\vm$ with $\rho,\varphi,G_M$ as:
\begin{equation*}
\vm=\rho G_M^{-1}\nabla\varphi.
\end{equation*}
Substituting $\vm$'s representation of $\rho,\varphi,G_M$ into (\ref{MFG-KKT2})(\ref{MFG-KKT3}), we obtain the MFG system:
\begin{subequations}
	\label{MFG-system}
	\begin{align}
	\label{MFG-system1}
	\rho_t+\nabla\cdot(\rho G_M^{-1}\nabla\varphi)&=0\\
	\label{MFG-system2}
	\frac{1}{2}\nabla\varphi^TG_{M}^{-1}\nabla\varphi+\varphi_t-\frac{\delta}{\delta\rho}\mathcal{F}(\rho)&=0,
	\end{align}
\end{subequations}
where (\ref{MFG-system1}) is the continuity equation, and (\ref{MFG-system2}) is the HJE . It describes the evolution of the velocity field.
\begin{remark}

	\rm Equation \eqref{MFG-system} represents the mean field limit of finite players' interaction system~\cite{cardaliaguet_master_2015,lasry_mean_2007}. Let  $X\sim \rho$ denote a flow map, and $P(t, x)=\nabla\varphi(t, x)$. Then the mean field game dynamics represents
	\begin{equation*}
	\left\{\begin{split}
	&\frac{dX}{dt}=G_M(X)^{-1}P\\
	&\frac{dP}{dt}=\nabla_X\Big(-\frac{1}{2}(P, G_M(X)^{-1}P)+\frac{\delta}{\delta\rho(X)}\mathcal{F}(\rho)\Big)
	\end{split}\right.
	\end{equation*}
	Here $\mathcal{F}(\rho)$ refers to the mean field limit of the interaction energy among players.

\end{remark}

\begin{remark}
	\rm Typical finite player games may involve noisy individual motions. Suppose that the players are affected by i.i.d. Brownian motion $\sqrt{2\beta}dB_t$, where $\beta>0$ is a given diffusion constant. In this circumstance, the continuity equation constraint in \ref{MFG-Benemou} is substituted by Fokker--Planck equation:
	$$\rho_t+\nabla \cdot (\rho\vv)=\beta\triangle\rho.$$
	The Laplacian depicts the viscosity among particles during transport. 
\end{remark}

Next, let us review two MFG examples. One is regularized optimal transport, and the other is the MFG with interaction energy. Our inverse-problem models are designed for the target parameters in each of them. 
\begin{example}[Optimal transport with regularization]
	\label{ex:OT}
	
	\rm
	For some $F$ convex, consider the running cost $\mathcal{F}(\rho(\cdot,t))$ as
	\begin{equation*}
	\mathcal{F}(\rho(\cdot,t))=\int_{\mathbb{T}^d}F(\rho(x,t))dx,
	\end{equation*}
	where $\vm$ is the flux. Reformulate (\ref{MFG-Benemou}) as:
	\begin{equation}
	\label{optimal-transport}
	\begin{split}
	\underset{\vm,\rho}{\mbox{minimize}}\quad
	&
	\displaystyle{
		\int_0^T\int_{\mathbb{T}^d} \frac{1}{2}\frac{\vm^TG_M\vm}{\rho}+F(\rho) dxdt
	}\\
	\mbox{subject to\quad}&\rho_t+\nabla \cdot \vm=0\\
	&\rho(\cdot,0)=\rho_0\quad \rho(\cdot,T)=\rho_T.
	\end{split}
	\end{equation}
	When $F=0$ and $G_M=I$, (\ref{optimal-transport}) is exactly the classical optimal transport problem.
	It transfers one pile of mass 
	to another with the least kinetic energy. The two piles have the same total mass but different density distributions. The minimum kinetic energy of the transport is called the Wasserstein $L^2$ distance between $\rho_0, \rho_T$.
	In~\cite{li_parallel_2018}, $F(\rho)=\epsilon\rho^2/2 $ serves as a regularization, giving rise to strong convexity for optimal transport. With non-zero $F$, we call (\ref{optimal-transport}) regularized optimal transport.
\end{example}
\begin{remark}
	\rm In optimal transport problem (\ref{optimal-transport}), sometimes, there exists a non-trivial ground metric function $G_M$. Here $G_M$ is a metric function on the sample space,
	which depicts a distance function between two sufficiently adjacent points on the manifold surface. {It is a metric function on the sample space, depicting the geometric contour of the surface.}
	So $G_M$ is an important parameter in MFG, and we will learn this ground metric function from observed population agents' (particles') motions.
\end{remark}

\begin{example}[MFG with interaction energy]
	\label{e.x.:MMD}
	\rm
	
	Take $\mathcal{F}(\cdot)$ as a convolution functional, also named interaction energy:
	\begin{equation}
	\label{F-convlution}
	\mathcal{F}(\rho(\cdot,t))=\int_{\mathbb{T}^d}\frac{1}{2}\rho(x,t) (K*\rho)(x,t)dx,
	\end{equation}
	where $K*\rho$ is the convolution of $K$ with $\rho$ defined as:
	\begin{equation*}
	(K*\rho)(x,t)=\int_{\mathbb{T}^d}K(x,y)\rho(y,t)dy.
	\end{equation*}
	Then (\ref{MFG-Benemou}) can be written as:
	\begin{equation}
	\label{MMD-control}
	\begin{split}
	\underset{\vm,\rho}{\mbox{minimize\quad}}
	&
	\displaystyle{
		\int_0^T\int_{\mathbb{T}^d} \frac{1}{2}\frac{\vm^TG_M\vm}{\rho}+\frac{1}{2}\rho(x,t) (K*\rho)(x,t) dxdt
	}\\
	\mbox{s.t.\quad}&
	\rho_t+\nabla \cdot \vm=0\\
	&\rho(\cdot,0)=\rho_0,\quad\rho(\cdot,T)=\rho_T.
	\end{split}
	\end{equation}
	The interaction energy $\mathcal{F}(\rho)$ in (\ref{MMD-control}) is related to the MMD, which is a widely used divergence (objective) functional in machine learning problems.
	In practice, $\rho_0$ is source data, $\rho_T$ indicates target data. $K$ is a convolution kernel. It induces symmetry as:
	\begin{equation}
	\label{Conv_ker}
	K(x,y)=K(y,x)=\tilde{K}(|y-x|_{\mathbb{T}^d}),\quad x,y\in \mathbb{T}^d,
	\end{equation}
	for some $\tilde{K}:[0,1/2]^d\rightarrow \mathbb{R}$. We define $|y_i-x_i|_{\mathbb{T}}$ as the distance between $x_i$ and $y_i$ on the 1-dimensional torus $\mathbb{T}$, $|y-x|_{\mathbb{T}^d}=(|y_1-x_1|_{\mathbb{T}},|y_2-x_2|_{\mathbb{T}},\ldots,|y_n-x_n|_{\mathbb{T}})^T$. In physical interpretation, $\mathcal{F}(\rho(\cdot,t))$ is the total potential energy attained by the particles in density distribution $\rho$.
	
	In practice, $\tilde{K}$ is often taken as:
	\begin{equation}
	\label{kernel-exp}
	\tilde{K}(\vx)=\exp\left(-\frac{\vx^TA\vx}{\epsilon}\right),\quad \vx\in\left[0,\frac{1}{2}\right]^d,
	\end{equation}
	where  $A\succ 0$ is called the adaptation matrix and $\epsilon>0$ is a scaling parameter.
\end{example}

The convolution kernel $K(\cdot,\cdot)$ depicts the pairwise impact between the particles, which relies on nothing but the relative distance of any two players. Once the kernel is fully studied, we are able to tell the interactions between the players while scheduling routine, thus predicting the dynamics of players under different time boundary conditions. In Section \ref{chap:MFG-model}, we propose inverse models for the above examples.

\section{Mean-field game inverse model}
\label{chap:MFG-model}
In this section, an equivalent MFG PDE system is proposed. Based on the PDEs, we develop two inverse-problem models to recreate the important parameters in MFGs.

\subsection{Mean-field game system}
\label{chapter:MFG-sys}
Consider the potential MFG problem (\ref{MFG-Benemou}) on the $d$-dimensional torus $\mathbb{T}^d$ and time interval $[0,T]$. We derive an equivalent system of MFG by Theorem \ref{thm:MFG-system}. Before displaying it, we introduce the following integral path for simplification of statement. Let $\hat{x}^i$ denote $x$ with the $i^{th}$ element erased, $\hat{x}^i=(x_1,\ldots,x_{i-1},\hat{x}_i,x_{i+1},\ldots,x_d)\in \mathbb{T}^{d-1}$. Path $s_i(\hat{x}^i,\cdot)$ is defined as:
\begin{equation*}
s_i(\hat{x}^i,s)=(x_1,x_2,\ldots,x_{i-1},s,x_{i+1},\ldots,x_n),\quad s\in[0,1].
\end{equation*}
\begin{theorem}
	\label{thm:MFG-system}
	Suppose $\mathcal{G}$ in (\ref{MFG-Benemou}) is an indicator function of $\{\rho_T\}$. Assume (\ref{MFG-Benemou}) has a positive solution $\rho\in\mathcal{P}_+(\mathbb{T}^d\times [0,T])$, and the dual variable $\varphi$ has continuous second-order mixed derivative in $\mathbb{T}^d\times[0,T]$:
	\begin{equation*}
	\varphi_{x_i,x_j},\,
	\varphi_{x_i,t}\in C(\mathbb{T}^d\times[0,T]),\quad i\not=j.
	\end{equation*}
	Then, $(\rho,\vv)$ is the minimizer of (\ref{MFG-Benemou}) if and only if it is the solution of PDEs
	\begin{subequations}
		\label{MFG-conserv2}
		\begin{align}
		&\rho_t+\nabla\cdot(\rho \vv)=0\\
		&(G_M\vv)_t+\nabla\left(\frac{1}{2}\vv^T G_M\vv-\frac{\delta}{\delta\rho}\mathcal{F}(\rho)\right)=0\\
		&\frac{\partial (G_M\vv)_i}{\partial{x_j}}=\frac{\partial (G_M\vv)_j}{\partial{x_i}}, \quad i\not=j\\
		&\int_{s_i(\hat{x}^i,\cdot)} (G_M\vv)_i dS_x=0,\quad\hat{x}^i\in\mathbb{T}^{d-1}, i=1,2,\ldots,d,
		\end{align}
	\end{subequations}
	where boundaries $\rho(\cdot,0),\rho(\cdot,T)$ are set to $\rho_0$ and $\rho_T$.
\end{theorem}

\begin{remark}
	\rm Theorem \ref{thm:MFG-system} claims the equivalence between (\ref{MFG-Benemou}) and (\ref{MFG-conserv2}). The deduction allows the MFG optimal solution to be represented by PDE constraints only with density distribution and velocity field $\rho,\vv$. Based on this formulation, we develop an inverse problem to recreate the  parameters, such as ground metrics or kernels, with feasible observations.
\end{remark}

Now let us give the explicit form of $\delta \mathcal{F}(\rho)/\delta\rho$ in specific cases. In Example \ref{ex:OT},
\begin{equation*}
\frac{\delta \mathcal{F}(\rho)}{\delta\rho}=F'(\rho).
\end{equation*}
In Example \ref{e.x.:MMD}, the $L^2$ variation to the MMD regularization can be written as:
\begin{equation*}
\frac{\delta \mathcal{F}(\rho)}{\delta\rho}=K*\rho,
\end{equation*}
which is the convolution of $K$ and $\rho$.    The rigorous proof of Theorem \ref{thm:MFG-system} is given in Subsection \ref{chap:proof-MFG-sys}. 

\subsection{Inverse model for ground metric}
We first look into the inverse model for Example \ref{ex:OT}. In this case, we are aiming at recreating ground metric $G_M(x)$, which only depends on the spatial location. For $G_M$, we further assume there exists a metric kernel $g_0:\mathbb{T}^d\rightarrow\mathbb{R}$ such that:
\begin{equation}
\label{Metric_ker}
G_M=(g_{ij})_{d\times d}=(f_{ij}(g_0))_{d\times d},
\end{equation}
with mappings $f_{ij}:\mathbb{R}\rightarrow\mathbb{R}, f_{ij}=f_{ji}$ given. Note that $f_{ij}$ are functions not explicitly depending on location $x$. The selection of $f_{ij}$ is flexible. They can be linear or non-linear functions. We have introduced $g_0$ to replace $G_M$, and this has reduced the dimension of the unknowns.
Now our inverse problem model is to learn the metric kernel $g_0$.
\begin{model}[Inverse Model for Ground Metric]
	\label{inv-ground-metric}
	\rm In (\ref{optimal-transport}), let $F:\mathcal{P}(\mathbb{T}^d)\rightarrow \mathbb{R}$ be known, and the ground $G_M$ can be represented by metric kernel $g_0$ as in (\ref{Metric_ker}). Then using the observations $\hat{\rho},\hat{\vv}$ of density distribution, velocity field, and the boundary observations $\hat{\rho}_0,\hat{\rho}_T$ from mean field games, we can define the following optimization problem:
	\begin{equation}
	\label{MFG-inv1}
	\begin{split}
	&\operatorname*{min}_{g_0,\rho,\vv}
	\frac{\alpha}{2}\|\rho-\hat{\rho}\|^2+\frac{\beta}{2}\|\vv-\hat{\vv}\|^2
	+\frac{\alpha_0}{2}
	\left(
	\|\rho_0-\hat{\rho}_0\|^2+\|\rho_T-\hat{\rho}_T\|^2
	\right)
	+\frac{\gamma}{p}\|\nabla g_0\|_p^p\\
	&\mbox{s.t.}
	\left\{
	\begin{split}
	&\rho_t+\nabla\cdot(\rho \vv)=0\\
	&(G_M\vv)_t+\nabla\left(\frac{1}{2}\vv^T G_M\vv-F'(\rho)\right)=0\\
	&\frac{\partial (G_M\vv)_i}{\partial{x_j}}=\frac{\partial (G_M\vv)_j}{\partial{x_i}}, \quad i\not=j\\
	&\int_{s_i(\hat{x}^i,\cdot)} (G_M\vv)_i dS_x=0,\quad \hat{x}^i\in\mathbb{T}^{d-1} ,i=1,2,\ldots,d
	\end{split}
	\right.
	\end{split}
	\end{equation}
	where $\|\cdot\|^2$ is the squared $L^2$ norm in the corresponding space ($\rho,\vv$ in $\mathbb{T}^d\times[0,T]$, while $\rho_0,\rho_T$ in $\mathbb{T}^d$). $\|\nabla g_0\|_p^p$ is the $H^p$ norm of the metric kernel $g_0$
	over $\mathbb{T}^d$, $p\geq1$. Parameters $\alpha,\alpha_0,\beta,\gamma$ are scaling indices.
\end{model}
\begin{remark}
	\rm In the model, the value of $p$ can be selected according to the property of $g_0$. If $g_0$ is smooth, we let $p=2$. When $g_0$ is sparse, we usually choose $p=1$, giving rise to the \emph{total variation} (TV) regularization. 
\end{remark}
\begin{remark}[Mean field observation]
	\rm In practice, there are also many alternative options for the objective function to regularize the distance between $(\rho,\vv)$ and $(\hat{\rho},\hat{\vv})$. For example, one can replace the $L^2$ norm of vector fields $\|\vv-\hat \vv\|^2$ by $\int\int\hat\rho(t,x) \|\vv(t,x)-\hat \vv(t,x)\|^2/2\,dxdt$. In this formulation, one only needs to fit the vector fields based on the current observation of density. Then, we replace the $L^2$ norm of density distribution $\|\rho-\hat{\rho}\|^2$ by Kullback-Leibler (KL) divergence, and the objective function becomes:
	\begin{equation}
	\label{kinetic-KL}
	\small{\begin{split}
		&\alpha\int_0^T\int_{\mathbb{T}^d}\rho \log\frac{\rho}{\hat{\rho}}\,dx dt+\int_0^T\int_{\mathbb{T}^d}\frac{\hat\rho}{2} \|\vv-\hat \vv\|^2\,dxdt
		+\alpha_0 \left(\int_{\mathbb{T}^d}\rho_0 \log\frac{\rho_0}{\hat{\rho}_0}\,dx
		+\int_{\mathbb{T}^d}\rho_T \log\frac{\rho_T}{\hat{\rho}_T}\,dx\right)+\frac{\gamma}{p}\|\nabla g_0\|_p^p.
		\end{split}}
	\end{equation}
	In this scheme,  $\int_{\mathbb{T}^d} \rho(x,t) dx=1$, where $\rho$ is a probability measure.
\end{remark}
\begin{remark}
	\rm It is also worth mentioning that there are other important regularizations in practice, for example, the $L^p$--Wasserstein metrics. 
\end{remark}
\subsection{Inverse model for convolution kernel}
In this section, we introduce another inverse model for Example \ref{e.x.:MMD}, which is to learn the kernel in interaction energy.

\begin{model}[Inverse Model for Convolution Kernel]
	\label{inv-MMD}
	\rm In (\ref{MMD-control}), suppose the ground metric $G_M$ is known. The convolution kernel is symmetric as in (\ref{Conv_ker}). The observations are density distribution $\hat{\rho}$, velocity field $\hat{\vv}$, and boundary observation $\hat{\rho}_0,\hat{\rho}_T$ from the numerical result of a single forward MFG problem. We design the inverse optimization model to learn the kernel $\tilde{K}$ over $[0,1/2]^d$:
	\begin{equation}
	\label{MFG-inv2}
	\begin{split}
	&\operatorname*{min}_{\tilde{K},\rho,\vv}
	\frac{\alpha}{2}\|\rho-\hat{\rho}\|^2+\frac{\beta}{2}\|\vv-\hat{\vv}\|^2
	+\frac{\alpha_0}{2}
	\left(
	\|\rho_0-\hat{\rho}_0\|^2+\|\rho_T-\hat{\rho}_T\|^2
	\right)
	+\frac{\gamma}{p}\|\nabla \tilde{K}\|_p^p\\
	&\mbox{s.t.}
	\left\{
	\begin{split}
	&\rho_t+\nabla\cdot(\rho \vv)=0\\
	&(G_M\vv)_t+\nabla\left(-K*\rho+\frac{1}{2}\vv^TG_M\vv \right)=0\\
	&\frac{\partial (G_M\vv)_i}{\partial{x_j}}=\frac{\partial (G_M\vv)_j}{\partial{x_i}}, \quad i\not=j\\
	&\int_{s_i(\hat{x}^i,\cdot)} (G_M\vv)_i dS_x=0,\quad \hat{x}^i\in\mathbb{T}^{d-1}, i=1,2,\ldots,d
	\end{split}
	\right.
	\end{split}
	\end{equation}
	where $\|\cdot\|^2$ shares the same definition as in (\ref{MFG-inv1}), and $\|\nabla \tilde{K}\|_p^p$ is the $H^p$ regularization on the kernel $\tilde{K}$ over $[0,1/2]^d$, $p\geq1, p\in\mathbb{Z}$.
\end{model}
\begin{remark}
	\rm We note that, due to the symmetry of $K$ on the torus, we simply study $\tilde{K}$ over $[0,1/2]^d$. Similar to Model \ref{inv-ground-metric}, $p$ can be selected as different positive integers based on the property of $K$. Since $\tilde{K}$ is often taken in the exponential quadratic format (\ref{kernel-exp}) in MFG interaction, we always assume $\tilde{K}$ is smooth as a priori and $p=2$ is a common choice for the regularization on the kernel. 
\end{remark}

\subsection{KKT condition for inverse model}
In this subsection, we show the KKT condition of the optimization problem (\ref{MFG-inv1}) in Model \ref{inv-ground-metric}.

We consider the optimization problem (\ref{MFG-inv1}) in 1-dimensional space, i.e. $d=1$. Taking $G_M=g_0$, $p=2$,  we have the theorem below.

\begin{theorem}
	Consider the optimization problem:
	\begin{subequations}
		\label{inverse-metric}
		\begin{gather}
		\label{inverse-metric1}
		\operatorname*{minimize}_{G_M,\rho,v}
		\mathcal{J}(\rho,v,\rho_0,\rho_T;\hat{\rho},\hat{v},\hat{\rho}_0,\hat{\rho}_T)
		+\frac{\gamma}{2}\| (G_M)_x\|_2^2\\
		\label{inverse-metric2}
		\mbox{s.t.}
		\left\{
		\begin{split}
		&\rho_t+(\rho v)_x=0\\
		&(G_M v)_t+\left(\frac{1}{2} G_M v^2-F'(\rho)\right)_x=0\\
		&\int_{\mathbb{T}}G_Mv\, dx=0
		\end{split}
		\right.
		\end{gather}
	\end{subequations}
	$\mathcal{J}$ is the regularization on $\rho,v$, as in objective function of (\ref{MFG-inv1}), or the regularization in (\ref{kinetic-KL}). Suppose  $\frac{\delta}{\delta \rho}\mathcal{J},
	\frac{\delta}{\delta v}\mathcal{J}$ exist, assume $\rho\in\mathcal{P}_+$ in the model. Denote the Lagrangian multiplier as $(\Phi,\psi)$. Then the minimizer $(\rho^*,v^*,G^*_M)$ of (\ref{MFG-inv1}) with certain multiplier $(\Phi^*,\psi^*)$ solves the following PDEs,
	\begin{subequations}
		\label{inverse-KKT}
		\begin{align}
		\label{inverse-KKT1}
		&\begin{cases}
		\frac{\delta}{\delta\rho}\mathcal{J}-\Phi_t-\Phi_x v+\psi_x F''(\rho)=0\\
		\frac{\delta}{\delta v}\mathcal{J}-\psi_t G_M-\Phi_x\rho-\psi_x G_Mv=0\\
		-\gamma (G_M)_{xx}-\int_0^T\psi_t v\,dt -\int_0^T\frac{1}{2}\psi_xv^2\,dt+\left.\psi v\right|_{t=0}^T=0
		\end{cases}
		\\
		\label{inverse-KKT2}
		&\begin{cases}
		\frac{\delta}{\delta\rho_0}\mathcal{J}-\Phi(x,0)=0\\
		\frac{\delta}{\delta\rho_T}\mathcal{J}+\Phi(x,T)=0
		\end{cases}
		\\
		\label{inverse-KKT3}
		&
		\psi(x,0)=\psi(x,T)=0
		\\
		\label{inverse-KKT4}
		& \mbox{constraints in (\ref{inverse-metric2})}.
		\end{align}
	\end{subequations}

\end{theorem}
\begin{proof}
	\rm The Lagrangian of (\ref{inverse-metric}) can be written as:
	\begin{equation*}
	\begin{split}
	&\operatorname*{min}_{\rho,v,G_M}\operatorname*{max}_{\Phi,\psi}
	\mathcal{J}(\rho,v,\rho_0,\rho_T;\hat{\rho},\hat{v},\hat{\rho}_0,\hat{\rho}_T)
	+\frac{\gamma}{2}\|(G_M)_x\|_2^2+
	\int_0^T\int_{\mathbb{T}}
	\Phi\left(
	(\rho v)_x+\rho_t
	\right)
	\\
	&\qquad
	+\psi
	\left(
	(G_M v)_t + \left(\frac{1}{2}G_M v^2-F'(\rho)\right)_x
	\right)\,dxdt\\
	=&\operatorname*{min}_{\rho,v,G_M}\operatorname*{max}_{\Phi,\psi}
	\mathcal{J}(\rho,v,\rho_0,\rho_T;\hat{\rho},\hat{v},\hat{\rho}_0,\hat{\rho}_T)
	+\frac{\gamma}{2}\|(G_M)_x\|_2^2+\int_0^T\int_{\mathbb{T}}
	-\Phi_x \rho v-\Phi_t \rho
	-\psi_t G_M v
	\\
	&\qquad
	-\psi_x
	\left(
	\frac{1}{2}G_M v^2-F'(\rho)
	\right)\,dxdt
	+
	\left.\int_{\mathbb{T}^d}\Phi\rho+\psi G_Mv\, dx
	\right|_{t=0}^T.
	\end{split}
	\end{equation*}
	By computing $L^2$ first variation to $\rho,v$ over $\mathbb{T}\times(0,T)$, $G_M$ over $\mathbb{T}$, we have (\ref{inverse-KKT1}). By computing derivative of $\rho_0,\rho_T$ over $\mathbb{T}\times\{0\},\mathbb{T}\times\{T\}$, (\ref{inverse-KKT2}) can be deduced. (\ref{inverse-KKT3}) comes from the derivative to $v(x,0),v(x,T)$ over $\mathbb{T}$. (\ref{inverse-KKT4}) is the original constraints.
\end{proof}

\subsection{Proof of Theorem \ref{thm:MFG-system}}
\label{chap:proof-MFG-sys}

\begin{proof}[Proof of Theorem \ref{thm:MFG-system}]
	\rm Once the solution to (\ref{MFG-Benemou}) is strictly positive, we have the KKT condition of (\ref{MFG-Benemou}) as (\ref{MFG-system}). Due to the convexity of (\ref{MFG-Benemou}) by taking $\vm=\rho\vv$ and the assumption that $\rho$ being positive, minimizer of the optimization problem equals to the solution of PDE system (\ref{MFG-system}). Thus our goal reduces to prove that (\ref{MFG-system}) and (\ref{MFG-conserv2}) share identical solutions under the continuous assumption of $\varphi$.

	Note $\vw=\nabla\varphi$. The HJE (\ref{MFG-system2}) can be written as:
	\begin{equation*}
	\begin{split}
	\frac{\delta}{\delta\rho}\mathcal{F}(\rho)-\frac{1}{2}\vw^TG_M^{-1}\vw=\varphi_t.
	\end{split}
	\end{equation*}
	Let
	\begin{equation*}
	\xi:=\frac{1}{2}\vw^TG_M^{-1}\vw-\frac{\delta}{\delta\rho}\mathcal{F}(\rho).
	\end{equation*}
	Since $\varphi$ has a second-order mixed derivative in $\mathbb{T}^d\times [0,T]$, we can substitute HJE with
	\begin{equation*}
	\vw_t+\nabla\xi=0,\quad \frac{\partial w_i}{\partial x_j}=\frac{\partial w_j}{\partial x_i},\quad i\not=j,
	\end{equation*}
	where $w_i,w_j$ are the $i^{th}$ and $j^{th}$ component of $\vw$. The MFG system is transformed into:
	\begin{equation}
	\label{MFG-conserv}
	\left\{
	\begin{split}
	&\rho_t+\nabla\cdot(\rho G_M^{-1}\vw)=0\\
	&\vw_t+\nabla\left(\frac{1}{2}\vw^TG_M^{-1}\vw-\frac{\delta}{\delta\rho}\mathcal{F}(\rho)\right)=0\\
	&\frac{\partial w_i}{\partial{x_j}}=\frac{\partial w_j}{\partial{x_i}}, \quad i\not=j
	\end{split}
	\right.
	\end{equation}
	Due to non-simple connection of the domain $\mathbb{T}^d$, for compatibility, the integral of $w_i$ on path $s_i(\hat{x}^i,\cdot)$ equals to $0$:
	\begin{equation*}
	\int_{s_i(\hat{x}^i,\cdot)} w_i dS_x=0,\quad \hat{x}^i\in\mathbb{T}^{d-1},i=1,2,\ldots,d.
	\end{equation*}
	Since
	\begin{equation*}
	\vv=\frac{\vm}{\rho}=G_M^{-1}\nabla\varphi=G_M^{-1}\vw,
	\end{equation*}
	substituting $\vw$ with $G_M\vv$, (\ref{MFG-conserv}) can be reformulated into (\ref{MFG-conserv2}). Thus $(\rho,\vv)$ is a solution to (\ref{MFG-system}) if and only if it solves (\ref{MFG-conserv2}). The theorem holds.
\end{proof}

\begin{remark}
	\rm We notice that if $\mathcal{G}$ is an indicator function, the forward problem forms the classical dynamical optimal transport problem, and our inverse model forms the inverse dynamical optimal transport. When $\mathcal{G}$ is a general functional, then our forward problem forms the classical mean field game problem, and our proposed model is the inverse mean field game problem. We also emphasize that above models share the same PDE system, expect for different boundary conditions on both initial and terminal time.
\end{remark}

\section{Discretization and rigorous treatment}
\label{chap:disc}
In this section, we derive the discrete format of the inverse problem in adaptation to the discrete forward problem. Furthermore, the primal-dual algorithm to solve the discrete inverse problem is provided.

\subsection{Discretization}
\label{chap:disc-treat}
In computation, we have to deal with the MFG problem in the discrete format. In this subsection, we use a finite volume discretization to approximate the continuous problem on the grid. Without loss of generality, all the work is done in 2 dimensions.

We discretize the problem on our dual variable $\varphi$. Consider an $m\times m\times n$ discretization on the torus $\mathbb{T}^2\times[0,T]$. Here $m\times m$ is the size for spatial discretization, while $n$ is the size for time discretization. We approximate the space-time domain $\mathbb{T}^d\times [0,T]$ with points $\{x_1,x_2,\ldots,x_m\}\times\{y_1,y_2,\ldots,y_m\}\times\{z_1,z_2,\ldots,z_n\}$. Take $\triangle x, \triangle t$ as the size of spatial-time element. Define the cube as:
\begin{equation*}
C(x,y,z)=\{(x',y')\in \mathbb{T}^2, z'\in[0,T] \mid |x-x'|_{\mathbb{T}}\leq \frac{\triangle x}{2},|y-y'|_{\mathbb{T}}\leq \frac{\triangle x}{2},|z-z'|\leq \frac{\triangle t}{2}
\}.
\end{equation*}
We further define the 2-dimensional box:
\begin{equation*}
C(x,y)=\{(x',y')\in \mathbb{T}^2 \mid |x-x'|_{\mathbb{T}}\leq \frac{\triangle x}{2},|y-y'|_{\mathbb{T}}\leq \frac{\triangle x}{2}
\}.
\end{equation*}
For simplicity, let $V$ denote the set of subscripts for all the grid points in space, which is modulo $m$ in the torus topology:
\begin{equation*}
V=\{1,2,\ldots m\}\times\{1,2,\ldots m\}.
\end{equation*}
Here $(i_1,i_2),(i'_1,i'_2)\in \mathbb{Z}^2$ represent an identical point in $V$ if $i_1\equiv i'_1(\mod m), i_2\equiv i'_2(\mod m)$. We define $\tilde{V}$ as:
\begin{equation*}
\tilde{V}=\{0,1,2,\ldots,\lfloor\frac{m}{2}\rfloor\}\times\{0,1,2,\ldots,\lfloor\frac{m}{2}\rfloor\},
\end{equation*}
which is our discretization of $[0,1/2]^2$.
Let $e_v$ be the unit vector in positive direction for each axis. In 2 dimensions, there are two such vectors: $e_1=(1,0), e_2=(0,1)$.
Slightly abusing the notation, we define:
\begin{align*}
&C(i,j)=C(x_{i_1},y_{i_2},z_j),\quad i\in V,j=1,2,\ldots,n,\\
&C(i,j-\frac{1}{2})=C(x_{i_1},y_{i_2},z_j-\frac{\triangle t}{2}),\quad i\in V,j=1,2,\ldots,n,\\
&C(i,n+\frac{1}{2})=C(x_{i_1},y_{i_2},z_n+\frac{\triangle t}{2}),\\
&C(i+\frac{e_v}{2},j)=C((x_{i_1},y_{i_2})+\triangle x e_v/2,z_j),\quad i\in V,j=1,2,\ldots,n,\\
&C(i)=C(x_{i_1},y_{i_2}),\quad i\in V,
\end{align*}
where $i=(i_1,i_2)$ contains 2 elements.

Next we define $\vm=(m_x,m_y)^T,\rho,\vv,\varphi,G_M$ in the discrete sense:
\begin{align*}
&\varphi_{i,j}=\frac{\int_{C(i,j)}\varphi(x,t)dxdt}{\mathrm{Vol}(C(i,j))},\quad i\in V,j=1,2,\ldots n,\\
&\rho_{i,j-\frac{1}{2}}=\frac{\int_{C(i,j-\frac{1}{2})}\rho(x,t)dxdt}{\mathrm{Vol}(C(i,j-\frac{1}{2}))},\quad i\in V,j=1,2,\ldots n+1,\\
&
\left(
\begin{array}{c}
m_{x,i+e_1/2,j}\\
m_{y,i+e_2/2,j}
\end{array}
\right)
=
\left(
\begin{array}{c}
\frac{\int_{C(i+e_1/2,j)} m_x(x,t)dxdt}{\mathrm{Vol}(C(i+e_1/2,j))}\\\begin{array}{c}
\frac{\int_{C(i+e_2/2,j)} m_y(x,t)dxdt}{\mathrm{Vol}(C(i+e_2/2,j))}
\end{array}
\end{array}
\right),
\quad i\in V,j=1,2,\ldots n,\\
&v_{x,i+e_1/2,j}=\frac{m_{x,i+e_1/2,j}}{\rho_{i,j-\frac{1}{2}}},\quad v_{y,i+e_2/2,j}=\frac{m_{y,i+e_2/2,j}}{\rho_{i,j-\frac{1}{2}}}, \quad i\in V,j=1,2,\ldots n,\\
&G_{M,i}=\frac{\int_{C(i)}G_M(x)dx}{\mathrm{Area}(C(i))},\quad i\in V,\quad K_{i,i'}=\frac{\int_{C(i)}\int_{C(i')}K(x,y)dxdy}{\mathrm{Area}(C(i))\cdot \mathrm{Area}(C(i'))},\quad i,i'\in V,
\end{align*}
where $\mathrm{Vol}(\cdot)$ is the volume of a cube, and $\mathrm{Area}(\cdot)$ is the area of a box.
Due to the symmetry of convolution kernel $K(\cdot,\cdot):\mathbb{T}^d\times\mathbb{T}^d\rightarrow \mathbb{R}$, it is safe for us to study $\tilde{K}(\cdot):\left[0,1/2\right]^d\rightarrow \mathbb{R}$ only. Let us denote the discretization of the single argument convolution kernel as:
\begin{equation}
\label{disc-K-tildeK}
\tilde{K}_{i}=K_{j,i+j},\quad i\in\tilde{V},j\in V.
\end{equation}
The selection of $j$ is arbitrary.
Furthermore, without causing confusion, we note
\begin{gather*}
m_{i+e_1/2,j}:=m_{x,i+e_1/2,j},\quad
m_{i+e_2/2,j}:=m_{y,i+e_2/2,j},\\
v_{i+e_1/2,j}:=v_{x,i+e_1/2,j},\quad
v_{i+e_2/2,j}:=v_{y,i+e_2/2,j},
\end{gather*}
to simplify the symbols.
We take the discrete format of $\rho_t$ and $\nabla\cdot\vm$ as:
\begin{align*}
(\rho_t)_{ij}=\left(\rho_{i,j+\frac{1}{2}}-\rho_{i,j-\frac{1}{2}}\right)/\triangle t,\quad
(\nabla\cdot\vm)_{ij}=\sum_{e_v}\left(m_{i+\frac{e_v}{2},j}-m_{i-\frac{e_v}{2},j}\right)/\triangle x.
\end{align*}
Set the discretization of $\int_0^T\mathcal{F}(\rho(\cdot,t))dt$ as
$\sum_{j=2}^n\mathcal{F}(\{\rho_{\cdot,j-\frac{1}{2}}\})\triangle t.$
Take $\mathcal{G}$ as indicator function. It can be relaxed into the constraint, in discrete format, $\rho_{i,n+\frac{1}{2}}=\rho_{T,i}, \quad i\in V$.
Finally, for notation simplicity, let $(\cdot)_v$ denote a vector whose elements take over all the positive direction $e_v$ in order. For example, suppose a variable $\eta_{i,j}$ is defined on the $d$-dimensional grid.
\begin{equation*}
(\eta_{i+e_v,j})_v=(\eta_{i+e_1,j},\eta_{i+e_2,j},\ldots,\eta_{i+e_d,j})^T.
\end{equation*}
Based on the discretization above, (\ref{MFG-Benemou}) can be written as:
\begin{equation}
\label{MFG-discrete}
\begin{split}
\underset{m,\rho}{\mbox{min}}\quad
&
\displaystyle{
	\sum_{i\in V}\sum_{j=1}^n
	\frac{1}{2}\frac{\left(m_{i+\frac{e_v}{2},j}\right)_v^TG_{M,i}\left(m_{i+\frac{e_v}{2},j}\right)_v}{\rho_{i,j-\frac{1}{2}}}+\sum_{j=2}^n\mathcal{F}(\{\rho_{\cdot,j-\frac{1}{2}}\})\frac{1}{\triangle x^2}
}\\
\mbox{s.t.\quad}&
\left(\rho_{i,j+\frac{1}{2}}-\rho_{i,j-\frac{1}{2}}\right)/\triangle t+\sum_{e_v}\left(m_{i+\frac{e_v}{2},j}-m_{i-\frac{e_v}{2},j}\right)/\triangle x=0\\
&\rho_{i,\frac{1}{2}}=\rho_{0,i},\quad\rho_{i,n+\frac{1}{2}}=\rho_{T,i},
\end{split}
\end{equation}
where $\{\rho_{0,\cdot}\},\{\rho_{T,\cdot}\}$ are given beforehand. We give explicit discretization of $\mathcal{F}(\rho(\cdot,t))$ for our examples. In (\ref{optimal-transport}), $\mathcal{F}(\rho(\cdot,t))$ is the integral of $F(\rho)$ and discretized as:
\begin{equation*}
\mathcal{F}(\{\rho_{\cdot,j-\frac{1}{2}}\})=
\sum_{i\in V}F(\rho_{i,j-\frac{1}{2}})\triangle x^2,\quad j=2,3,\ldots,n.
\end{equation*}
In (\ref{MMD-control}), $\mathcal{F}(\cdot)$ is a convolution function of $\rho$, whose discrete format can be written as:
\begin{equation*}
\mathcal{F}(\{\rho_{\cdot,j-\frac{1}{2}}\})=
\sum_{i\in V}\sum_{i'\in V}
\frac{1}{2}
K(i,i')
\rho_{i,j-\frac{1}{2}}
\rho_{i',j-\frac{1}{2}}\triangle x^4,\quad j=2,3,\ldots,n.
\end{equation*}


\subsection{Discrete mean-field game system}
\label{chap:MFGs-disc}
With the discretization in Subsection \ref{chap:disc-treat}, we derive a discrete format of inverse problem (\ref{MFG-conserv2}) compatible with the discrete potential MFG (\ref{MFG-discrete}).

\begin{theorem}
	\label{thm:MFG-disc}
	Suppose that (\ref{MFG-discrete}) has a strictly positive solution $\{\rho_{i,j-\frac{1}{2}}\}$.
	Then, the tuple $(\{\rho_{i,j-\frac{1}{2}}\},\{m_{i+\frac{e_v}{2},j}\})$ is a minimizer of (\ref{MFG-discrete}) if and only if the corresponding tuple $(\{\rho_{i,j-\frac{1}{2}}\},\{v_{i+\frac{e_v}{2},j}\})$ is the solution of following discrete MFG system:
	\begin{equation}
	\label{MFG-conserv-disc}
	\left\{
	\begin{split}
	&\left(\frac{\xi_{i+e_v,j-\frac{1}{2}}-\xi_{i,j-\frac{1}{2}}}{\triangle x}\right)_v+\left(\frac{(w_{i+\frac{e_v}{2},j})_v-(w_{i+\frac{e_v}{2},j-1})_v}{\triangle t}\right)=0,\, i\in V,j=2,3,\ldots,n\\
	&\left(\rho_{i,j+\frac{1}{2}}-\rho_{i,j-\frac{1}{2}}\right)/\triangle t+\sum_{e_v}\left(m_{i+\frac{e_v}{2},j}-m_{i-\frac{e_v}{2},j}\right)/\triangle x=0,\, i\in V,j=1,2,\ldots,n\\
	&\frac{w_{i+\frac{e_v}{2}+e_w,j}-w_{i+\frac{e_v}{2},j}}{\triangle x}=\frac{w_{i+\frac{e_w}{2}+e_v,j}-w_{i+\frac{e_w}{2},j}}{\triangle x},\quad e_v\not=e_w,i\in V,j=1,2,\ldots,n\\
	&
	\sum_{i_1}w_{i+\frac{e_1}{2},j}=0\quad i_2=1,2,\ldots,m,\quad \sum_{i_2}w_{i+\frac{e_2}{2},j}=0\quad i_1=1,2,\ldots,m,
	j=1,2\ldots,n
	\end{split}
	\right.
	\end{equation}
	with initial and terminal states set as:
	\begin{equation*}
	\rho_{i,\frac{1}{2}}=\rho_{0,i},\quad\rho_{i,n+\frac{1}{2}}=\rho_{T,i}.
	\end{equation*}
	In (\ref{MFG-conserv-disc}),
	\begin{equation*}
	\begin{split}
	&(w_{i+\frac{e_v}{2},j})_v=G_{M,i}(v_{i+\frac{e_v}{2},j})_v,
	\quad i\in V, j=1,2,\ldots n,\\
	&\xi_{i,j-\frac{1}{2}}=
	\frac{1}{2}(v_{i+\frac{e_v}{2},j})_v^TG_{M,i}(v_{i+\frac{e_v}{2},j})_v
	-
	\frac{\partial}{\partial \rho_{i}}\mathcal{F}(\{\rho_{\cdot,j-\frac{1}{2}}\})
	\frac{1}{\triangle x^2}.
	\end{split}
	\end{equation*}
\end{theorem}
\begin{proof}[Proof of Theorem \ref{thm:MFG-disc}]
	\rm We start with the discrete potential MFG (\ref{MFG-discrete}), whose Lagrangian can be formulated as:
	\begin{align}
	\label{Lagrangian-discrete1}
	\begin{split}
	&\operatorname*{min}_{m,\rho}\operatorname*{max}_{\varphi}
	\sum_{i\in V}\sum_{j=1}^n
	\left\{\frac{1}{2}\frac{(m_{i+\frac{e_v}{2},j})_v^TG_{M,i}(m_{i+\frac{e_v}{2},j})_v}{\rho_{i,j-\frac{1}{2}}}+
	\varphi_{i,j}\left(
	\left(\rho_{i,j+\frac{1}{2}}-\rho_{i,j-\frac{1}{2}}\right)/\triangle t\right.\right.\\
	&\qquad\left.\left.+\sum_{e_v}\left(m_{i+\frac{e_v}{2},j}-m_{i-\frac{e_v}{2},j}\right)/\triangle x
	\right)\right\}+
	\sum_{j=2}^n\mathcal{F}(\{\rho_{\cdot,j-\frac{1}{2}}\})\frac{1}{\triangle x^2}
	\end{split}\\
	\label{Lagrangian-discrete2}
	\begin{split}
	=&\operatorname*{min}_{m,\rho}\operatorname*{max}_{\varphi}
	\sum_{i\in V}\sum_{j=1}^n
	\frac{1}{2}\frac{(m_{i+\frac{e_v}{2},j})_v^TG_{M,i}(m_{i+\frac{e_v}{2},j})_v}{\rho_{i,j-\frac{1}{2}}}+
	\sum_{i\in V}\sum_{j=2}^n\frac{\varphi_{i,j-1}-\varphi_{i,j}}{\triangle t}\rho_{i,j-\frac{1}{2}}
	+
	\sum_{i\in V}\frac{\varphi_{i,n}}{\triangle t}\rho_{i,n+\frac{1}{2}}\\
	&\qquad
	-
	\sum_{i\in V}
	\frac{\varphi_{i,1}}{\triangle t}
	\rho_{i,\frac{1}{2}}
	+
	\sum_{i\in V}\sum_{j=1}^n
	\left(
	\frac{\varphi_{i,j}-\varphi_{i+e_v,j}}{\triangle x}
	\right)_v^T(m_{i+\frac{e_v}{2},j})_v
	+
	\sum_{j=2}^n\mathcal{F}(\{\rho_{\cdot,j-\frac{1}{2}}\})\frac{1}{\triangle x^2}.
	\end{split}
	\end{align}
	In (\ref{Lagrangian-discrete1}) and (\ref{Lagrangian-discrete2}), $\{\rho_{i,\frac{1}{2}}\}$ and $\{\rho_{i,n+\frac{1}{2}}\}$ are fixed.
	
	Taking derivative to the unknowns $\rho_{i,j-\frac{1}{2}},\,m_{i+\frac{e_v}{2},j}$ and $\varphi_{i,j}$, we obtain the discrete KKT condition:
	\begin{equation}
	\label{MFG-KKT-disc}
	\left\{
	\begin{split}
	&\frac{G_{M,i}(m_{i+\frac{e_v}{2},j})_v}{\rho_{i,j-\frac{1}{2}}}+\left(
	\frac{\varphi_{i,j}-\varphi_{i+e_v,j}}{\triangle x}
	\right)_v=0,
	\quad i\in V, j=1,2,\ldots n\\
	&-\frac{1}{2}\frac{(m_{i+\frac{e_v}{2},j})_v^TG_{M,i}(m_{i+\frac{e_v}{2},j})_v}{{\rho_{i,j-\frac{1}{2}}}^2}
	+
	\frac{\varphi_{i,j-1}-\varphi_{i,j}}{\triangle t}
	+
	\frac{\partial}{\partial \rho_{i,j-\frac{1}{2}}}\mathcal{F}(\{\rho_{\cdot,j-\frac{1}{2}}\})\frac{1}{\triangle x^2}=0,\\
	&\qquad i\in V,j=2,3,\ldots,n\\
	&\left(\rho_{i,j+\frac{1}{2}}-\rho_{i,j-\frac{1}{2}}\right)/\triangle t+\sum_{e_v}\left(m_{i+\frac{e_v}{2},j}-m_{i-\frac{e_v}{2},j}\right)/\triangle x=0,\, i\in V,j=1,2,\ldots,n
	\end{split}
	\right.
	\end{equation}
	whose positive solution is the minimizer of (\ref{MFG-discrete}). Solving the first equation in (\ref{MFG-KKT-disc}), we are able to get: 
	\begin{equation*}
	(m_{i+\frac{e_v}{2},j})_v=\rho_{i,j-\frac{1}{2}}G_{M,i}^{-1}\left(
	\frac{\varphi_{i+e_v,j}-\varphi_{i,j}}{\triangle x}
	\right)_v,
	\quad i\in V, j=1,2,\ldots n.
	\end{equation*}
	Let
	\begin{equation*}
	(w_{i+\frac{e_v}{2},j})_v=\left(
	\frac{\varphi_{i+e_v,j}-\varphi_{i,j}}{\triangle x}
	\right)_v=G_{M,i}(v_{i+\frac{e_v}{2},j})_v,
	\quad i\in V, j=1,2,\ldots n.
	\end{equation*}
	Substituting the solved $(m_{i+\frac{e_v}{2},j})_v$ into the second and third equation of (\ref{MFG-KKT-disc}), we have:
	\begin{equation}
	\label{MFG-system-disc}
	\left\{
	\begin{split}
	&-\frac{1}{2}(w_{i+\frac{e_v}{2},j})_v^TG^{-1}_{M,i}(w_{i+\frac{e_v}{2},j})_v
	+
	\frac{\varphi_{i,j-1}-\varphi_{i,j}}{\triangle t}
	+
	\frac{\partial}{\partial \rho_{i}}\mathcal{F}(\{\rho_{\cdot,j-\frac{1}{2}}\})\frac{1}{\triangle x^2}=0,\\
	&\qquad i\in V,j=2,3,\ldots,n\\
	&\left(\rho_{i,j+\frac{1}{2}}-\rho_{i,j-\frac{1}{2}}\right)/\triangle t+\sum_{e_v}\left(m_{i+\frac{e_v}{2},j}-m_{i-\frac{e_v}{2},j}\right)/\triangle x=0,\, i\in V,j=1,2,\ldots,n\\
	&(m_{i+\frac{e_v}{2},j})_v=\rho_{i,j-\frac{1}{2}}G_{M,i}^{-1}(w_{i+\frac{e_v}{2},j})_v,
	\quad i\in V, j=1,2,\ldots n
	\end{split}
	\right.
	\end{equation}
	Take
	\begin{equation*}
	\xi_{i,j-\frac{1}{2}}=
	\frac{1}{2}(w_{i+\frac{e_v}{2},j})_v^TG^{-1}_{M,i}(w_{i+\frac{e_v}{2},j})_v
	-
	\frac{\partial}{\partial \rho_{i}}\mathcal{F}(\{\rho_{\cdot,j-\frac{1}{2}}\})
	\frac{1}{\triangle x^2}.
	\end{equation*}
	Then the equivalent equations	 derived from mixed second-order discrete derivative of $\varphi$ can be written as:
	\begin{equation}
	\label{MFG-discrete-p1}
	\left\{
	\begin{split}
	&\left(\frac{\xi_{i+e_v,j-\frac{1}{2}}-\xi_{i,j-\frac{1}{2}}}{\triangle x}\right)_v+\left(\frac{(w_{i+\frac{e_v}{2},j})_v-(w_{i+\frac{e_v}{2},j-1})_v}{\triangle t}\right)=0,\, i\in V,j=2,3,\ldots,n\\
	&\frac{w_{i+\frac{e_v}{2}+e_w,j}-w_{i+\frac{e_v}{2},j}}{\triangle x}=\frac{w_{i+\frac{e_w}{2}+e_v,j}-w_{i+\frac{e_w}{2},j}}{\triangle x},\quad e_v\not=e_w,i\in V,j=1,2,\ldots,n
	\end{split}
	\right.
	\end{equation}
	The second equation in \eqref{MFG-KKT-disc} can be substituted with (\ref{MFG-discrete-p1}). Besides, due to $\mathbb{T}^d$'s non-simply connection, for $i=(i_1,i_2)$, we add another constraint for compatibility:
	\begin{equation}
	\label{MFG-discrete-p2}
	\begin{split}
	&\sum_{i_1}w_{i+\frac{e_1}{2},j}=0,\quad i_2=1,2,\ldots,m, j=1,2,\ldots,n\\ &\sum_{i_2}w_{i+\frac{e_2}{2},j}=0,\quad i_1=1,2,\ldots,m, j=1,2,\ldots,n.
	\end{split}
	\end{equation}
	Combining (\ref{MFG-discrete-p1}), (\ref{MFG-discrete-p2}), and the first equation in (\ref{MFG-system-disc}), we get (\ref{MFG-conserv-disc}). Thus the minimizer of (\ref{MFG-discrete}) equals to the solution to (\ref{MFG-conserv-disc}). The theorem holds.
\end{proof}

\begin{remark}
	\rm In (\ref{MFG-conserv-disc}), $\{\xi_{i,j-\frac{1}{2}}\},\{m_{i+\frac{e_v}{2},j}\},\{w_{i+\frac{e_v}{2},j}\}$ can be totally represented by discrete $\rho,\vw$. Here (\ref{MFG-conserv-disc}) is a system purely consisting of discrete $\rho,\vv$ and transport parameters, $G_M$,$K$, etc. Taking (\ref{MFG-conserv-disc}) as constraints, we are able to design a discrete inverse problem to learn target parameters with feasible observations of $\rho,\vv$.
\end{remark}

\subsection{Inverse problem in discrete format}

With the discretization in Subsection \ref{chap:disc-treat} and discrete MFG system in Subsection \ref{chap:MFGs-disc}, we are able to design a discrete format inverse model to learn the metric kernel.
\begin{model}[Discrete Inverse Model for Ground Metric]
	\label{disc-inv-ground-metric}
	\rm For the regularized optimization problem (\ref{optimal-transport}), assume that the running cost functional $\mathcal{F}$ is known beforehand, and suppose the observation of density distribution $\{\hat{\rho}_{i,j-\frac{1}{2}}\},\,i\in V,j=1,2,\ldots,n+1$ and velocity field $\{(\hat{v}_{i+\frac{e_v}{2},j})_v\},\,i\in V,j=1,2,\ldots,n$ is available.
	In the observation, $\{\hat{\rho}_{\cdot,\frac{1}{2}}\}$ and $\{\hat{\rho}_{\cdot,n+\frac{1}{2}}\}$ denote the initial/terminal discrete density distribution.
	Then we have an inverse model on the grid as follows:
	\begin{equation}
	\label{MFG-inv-disc1}
	\begin{split}
	&\operatorname*{min}_{g_0,\rho,v}
	\sum_{i\in V}\sum_{j=1}^{n+1}\frac{\alpha}{2}\left(\rho_{i,j-\frac{1}{2}}-\hat{\rho}_{i,j-\frac{1}{2}}\right)^2+
	\sum_{i\in V}\sum_{j=1}^n\frac{\beta}{2}\|(v_{i+\frac{e_v}{2},j})_v-(\hat{v}_{i+\frac{e_v}{2},j})_v\|^2
	\\
	&\quad+\sum_{i\in V}
	\frac{\alpha_0}{2\triangle t}
	\left(
	\left(\rho_{i,\frac{1}{2}}-\hat{\rho}_{i,\frac{1}{2}}\right)^2
	+
	\left(\rho_{i,n+\frac{1}{2}}-\hat{\rho}_{i,n+\frac{1}{2}}\right)^2
	\right)
	+\sum_{i\in V}\sum_{e_v}\frac{\gamma}{p\triangle t}
	\left|\frac{g_{0,i+e_v}-g_{0,i}}{\triangle x}\right|^p
	\\
	&s.t.
	\left\{
	\begin{split}
	&\left(\frac{\xi_{i+e_v,j-\frac{1}{2}}-\xi_{i,j-\frac{1}{2}}}{\triangle x}\right)_v+\left(\frac{G_{M,i}(v_{i+\frac{e_v}{2},j})_v-G_{M,i}(v_{i+\frac{e_v}{2},j-1})_v}{\triangle t}\right)=0,\\
	&\qquad 
	i\in V,j=2,3,\ldots,n\\
	&\left(\rho_{i,j+\frac{1}{2}}-\rho_{i,j-\frac{1}{2}}\right)/\triangle t+\sum_{e_v}\left(m_{i+\frac{e_v}{2},j}-m_{i-\frac{e_v}{2},j}\right)/\triangle x=0,\, i\in V,j=1,2,\ldots,n\\
	&\frac{w_{i+\frac{e_v}{2}+e_w,j}-w_{i+\frac{e_v}{2},j}}{\triangle x}=\frac{w_{i+\frac{e_w}{2}+e_v,j}-w_{i+\frac{e_w}{2},j}}{\triangle x},\, e_v\not=e_w,i\in V,j=1,2,\ldots,n\\
	&
	\sum_{i_1}w_{i+\frac{e_1}{2},j}=0,\quad i_2=1,2,\ldots,m,\quad \sum_{i_2}w_{i+\frac{e_2}{2},j}=0,\quad i_1=1,2,\ldots,m,\,
	j=1,2\ldots,n
	\end{split}
	\right.
	\end{split}
	\end{equation}
	where
	\begin{equation*}
	\left\{
	\begin{split}
	&\xi_{i,j-\frac{1}{2}}=
	\frac{1}{2}(v_{i+\frac{e_v}{2},j})_v^TG_{M,i}(v_{i+\frac{e_v}{2},j})_v
	-
	F'(\rho_{i,j-\frac{1}{2}}),\quad i\in V,j=2,3,\ldots,n\\
	&
	(w_{i+\frac{e_v}{2},j})_v=G_{M,i}(v_{i+\frac{e_v}{2},j})_v,\quad i\in V,j=1,2,\ldots,n\\
	&
	(m_{i+\frac{e_v}{2},j})_v=\rho_{i,j-\frac{1}{2}}(v_{i+\frac{e_v}{2},j})_v,\quad i\in V,j=1,2,\ldots,n
	\end{split}
	\right.
	\end{equation*}
	Note that in (\ref{MFG-inv-disc1}), discrete $\xi,\vw,\vm$ are auxiliary variables for the simplicity of statement, and they can be totally represented by discrete $\rho,\vv,g_0$, the unknowns in the inverse optimization problem. The selection of $p$ is also flexible as in continuous model (\ref{MFG-inv2}). The model is adaptive to the prior properties of $g_0$.
\end{model}

Similar to Model \ref{disc-inv-ground-metric}, following the idea in Model \ref{inv-MMD} and applying the discrete system (\ref{MFG-conserv-disc}), we develop a discrete inverse problem to learn the convolution kernel:
\begin{model}[Discrete Inverse Model for Convolution Kernel]
	\label{disc-inv-kernel}
	\rm In the interaction energy regularized MFG (\ref{MMD-control}), we assume the ground metric $G_M$ is known, and the observations $\{\hat{\rho}_{i,j-\frac{1}{2}}\},\,i\in V,j=1,2,\ldots,n+1$ for  density distribution and  $\{(\hat{v}_{i+\frac{e_v}{2},j})_v\},\,i\in V,j=1,2,\ldots,n$ for velocity field during the game are given. Here $\{\hat{\rho}_{\cdot,\frac{1}{2}}\}, \{\hat{\rho}_{\cdot,n+\frac{1}{2}}\}$, standing for the initial/terminal states observation, are also known. With the information above, we can construct an inverse model to learn the convolution kernel $\tilde{K}$ as follows:
	\begin{equation}
	\label{MFG-inv-disc2}
	\begin{split}
	&\operatorname*{min}_{\tilde{K},\rho,v}
	\sum_{i\in V}\sum_{j=1}^{n+1}\frac{\alpha}{2}\left(\rho_{i,j-\frac{1}{2}}-\hat{\rho}_{i,j-\frac{1}{2}}\right)^2
	+
	\sum_{i\in V}\sum_{j=1}^n\frac{\beta}{2}\|(v_{i+\frac{e_v}{2},j})_v-(\hat{v}_{i+\frac{e_v}{2},j})_v\|^2
	\\
	&\qquad
	+
	\sum_{i\in V}
	\frac{\alpha_0}{2\triangle t}
	\left(
	\left(\rho_{i,\frac{1}{2}}-\hat{\rho}_{i,\frac{1}{2}}\right)^2
	+
	\left(\rho_{i,n+\frac{1}{2}}-\hat{\rho}_{i,n+\frac{1}{2}}\right)^2
	\right)
	+
	\sum_{i\in \tilde{V}}\sum_{\{e_v|i+e_v\in\tilde{V}\}}\frac{\gamma}{p\triangle t}
	\left|\frac{\tilde{K}_{i+e_v}-\tilde{K}_{i}}{\triangle x}\right|^p
	\\
	&s.t.
	\left\{
	\begin{split}
	&\left(\frac{\xi_{i+e_v,j-\frac{1}{2}}-\xi_{i,j-\frac{1}{2}}}{\triangle x}\right)_v+\left(\frac{G_{M,i}(v_{i+\frac{e_v}{2},j})_v-G_{M,i}(v_{i+\frac{e_v}{2},j-1})_v}{\triangle t}\right)=0,\\
	&\qquad i\in V,j=2,3,\ldots,n\\
	&\left(\rho_{i,j+\frac{1}{2}}-\rho_{i,j-\frac{1}{2}}\right)/\triangle t+\sum_{e_v}\left(m_{i+\frac{e_v}{2},j}-m_{i-\frac{e_v}{2},j}\right)/\triangle x=0,\quad i\in V,j=1,2,\ldots,n\\
	&\frac{w_{i+\frac{e_v}{2}+e_w,j}-w_{i+\frac{e_v}{2},j}}{\triangle x}=\frac{w_{i+\frac{e_w}{2}+e_v,j}-w_{i+\frac{e_w}{2},j}}{\triangle x},\, e_v\not=e_w,i\in V,j=1,2,\ldots,n\\
	&
	\sum_{i_1}w_{i+\frac{e_1}{2},j}=0,\quad i_2=1,2,\ldots,m,\quad \sum_{i_2}w_{i+\frac{e_2}{2},j}=0,\quad i_1=1,2,\ldots,m, j=1,2\ldots,n
	\end{split}
	\right.
	\end{split}
	\end{equation}
	where
	\begin{equation*}
	\left\{
	\begin{split}
	&\xi_{i,j-\frac{1}{2}}=
	\frac{1}{2}(v_{i+\frac{e_v}{2},j})_v^TG_{M,i}(v_{i+\frac{e_v}{2},j})_v
	-
	\sum_{i'\in V}K(i,i')\rho_{i',j-\frac{1}{2}}\triangle x^2
	,\quad i\in V,j=2,3,\ldots,n\\
	&
	(w_{i+\frac{e_v}{2},j})_v=G_{M,i}(v_{i+\frac{e_v}{2},j})_v,\quad i\in V,j=1,2,\ldots,n\\
	&
	(m_{i+\frac{e_v}{2},j})_v=\rho_{i,j-\frac{1}{2}}(v_{i+\frac{e_v}{2},j})_v,\quad i\in V,j=1,2,\ldots,n
	\end{split}
	\right.
	\end{equation*}
	The relationship between discrete $K$ and $\tilde{K}$ is shown in (\ref{disc-K-tildeK}). Here (\ref{MFG-inv-disc2}) is a model learning the convolution kernel in MFG with interaction energy with fully observation of $\rho,\vv$. We target on solving unknown discrete variables $\rho,\vv,\tilde{K}$ in the optimization.
\end{model}
\begin{remark}
	\rm The four constraints of (\ref{MFG-inv-disc1}) and (\ref{MFG-inv-disc2}) are in  discrete differential formats. The first and the third constraints are bi-linear in $\rho,\vv,G_M$ or $\tilde{K}$, while the second one is close to being bi-linear except a quadratic term in $\xi$.
	Furthermore, we notice that, once the Lagrangian is deduced, the difference operator can be transferred to the dual variable via integration by part.
\end{remark}

\subsection{Algorithm}
\label{chap:disc-alg}
In this section, we present the algorithm for solving the optimization problems in Model \ref{disc-inv-ground-metric} and Model \ref{disc-inv-kernel}.

We apply a primal-dual algorithm to solve the discrete optimization problem. To better illustrate our algorithm, we reformulate the optimization problem as:
\begin{gather*}
\operatorname*{minimize}_{\rho,\vv,\theta}f(\rho,\vv,\theta)\\
s.t.
\quad
c_i(\rho,\vv,\theta)=0,\quad i=1,2,\ldots,r,
\end{gather*}
where $f(\rho,\vv,\theta)$ is the objective function, $\theta$ is the parameter to be recreated, which is $g_0$ in Model \ref{disc-inv-ground-metric} and $\tilde{K}$ in Model \ref{disc-inv-kernel}. Here $c_i$ is constraint function, and $r$ indicates the number of constraints. In both of the models, we have 4 constraints.

Write the Lagrangian:
\begin{equation*}
L(\rho,\vv,\theta;\{\psi_i\}_{i=1}^r)=f(\rho,\vv,\theta)+\sum_{i=1}^{r}c_{i}(\rho,\vv,\theta)\psi_i,
\end{equation*}
where $\psi_i$ stands for the dual variables or Lagrangian multipliers. In the primal step, we fix the dual variable and update the primal one with gradient descent:
\begin{equation}
\label{primal_update}
\left\{
\begin{split}
&\rho^{k+1}=\rho^{k}-\tau_\rho \frac{\delta}{\delta \rho}L(\rho^k,\vv^k,\theta^k;\{\psi_i^k\}_{i=1}^r)\\
&\vv^{k+1}=\vv^{k}-\tau_\vv \frac{\delta}{\delta \vv}L(\rho^k,\vv^k,\theta^k;\{\psi_i^k\}_{i=1}^r)\\
&\theta^{k+1}=\theta^{k}-\tau_\theta \frac{\delta}{\delta \theta}L(\rho^k,\vv^k,\theta^k;\{\psi_i^k\}_{i=1}^r)\\
\end{split}
\right.
\end{equation}
The parameter $\tau$ is the descend step size. In practice, we sometimes take step length $\tau_\rho,\tau_\vv,\tau_\theta$ as different values, thus getting a better convergence rate. Next, we update the dual variable while fixing the newly updated primal variables:
\begin{equation}
\label{dual_update}
\psi_i^{k+1}=\psi_i^{k}+\sigma c_i(\rho^*,\vv^*,\theta^*),\quad i=1,2,\ldots,r,
\end{equation}
where $\rho^*,\vv^*,\theta^*$ are defined as:
\begin{equation}
\label{dual_star}
\rho^*=2\rho^{k+1}-\rho^k,\quad \vv^*=2\vv^{k+1}-\vv^k,\quad \theta^*=2\theta^{k+1}-\theta^k.
\end{equation}
Also $\sigma$ is the step size. This follows the format in primal-dual hybrid (PDHG) algorithm~\cite{chambolle_ergodic_2016}. The pseudo-code is shown in Algorithm \ref{alg:primal-dual}. The detailed iteration steps for Model \ref{disc-inv-ground-metric} are given in Appendix \ref{apped:alg}.

\begin{algorithm}
	\caption{Primal-Dual Algorithm}
	\label{alg:primal-dual}
	\begin{algorithmic}
		\STATE{\textbf{Input:} the observations $\hat{\rho},\hat{\vv}$, all the scaling parameters $\alpha,\alpha_0,\beta,\gamma$, the norm index $p$, and other prior known parameters in the inverse model. Input the iteration parameters $\tau_\rho,\tau_\vv,\tau_\theta,\sigma$.}
		\STATE
		{\textbf{Initialization:} Set $\rho>0$, $\vv=0$, $\theta>0$.}
		\FOR{$k=1,2,\ldots$ \quad(until convergence)}
		\STATE{Primal step:}
		\STATE{Update $\rho^{k+1},\vv^{k+1},\theta^{k+1}$ as in (\ref{primal_update}).}
		\STATE{Dual step:}
		\STATE{Update $\rho^*,\vv^*,\theta^*$ in the dual iteration as in (\ref{dual_star}).}
		\STATE{Update dual variables $\psi_i^{k+1}$ as in (\ref{dual_update}).}
		\ENDFOR
	\end{algorithmic}
\end{algorithm}

\begin{remark}
	\rm
	The spatial discretization and primal-dual iteration make the algorithm parallelizable. Both primal and dual updates reduce to parallel subproblems, which enables high computational efficiency.
\end{remark}

Typically, $\alpha,\beta$ can always be taken as $1/\|\hat{\rho}\|^2,1/\|\hat{\vv}\|^2$. However, there is little known for the selection of $\gamma$ in (\ref{intro-model}). To avoid the choice of $\gamma$, we can apply Bregman iteration~\cite{OsherBurgerGoldfarbXuYin2005_iterative,yin_bregman_2008}. We can start with a sub-optimal choice of $\gamma$ and get a better recreation of the target parameter $\theta$ after some iterations. We consider problem (\ref{intro-model}). The idea of Bregman iteration is as follows. Let
$$J(\theta)=\frac{\gamma}{p}\|\nabla\theta\|_p^p,$$
and the Bregman divergence between $\theta$ and $\tilde{\theta}$ is defined as
$$D_J^\vq(\theta,\tilde{\theta})=J(\theta)-J(\tilde{\theta})-\langle \vq,\theta-\tilde{\theta}\rangle,$$
with some $\vq\in\partial J(\tilde{\theta})$ being a subgradient for $J$ at $\tilde{\theta}$. Set $\vq^0=0$. In each Bregman iteration, we solve the following optimization problem:
\begin{equation}
\label{Breg_iter}
\begin{split}
\underset{\theta,\rho,\vv}{\mbox{min}}\quad
&
\displaystyle{
	D_J^{\vq^l}(\theta,\theta^l)+
	\frac{\alpha}{2}\|\rho-\hat{\rho}\|^2+\frac{\beta}{2}\|\vv-\hat{\vv}\|^2
	+\frac{\alpha_0}{2}
	\left(
	\|\rho_0-\hat{\rho}_0\|^2+\|\rho_T-\hat{\rho}_T\|^2
	\right)
}\\
\mbox{s.t.\quad}
&c_i(\rho,\vv,\theta)=0.\quad i=1,2,\ldots,r.
\end{split}
\end{equation}	
By applying the primal-dual algorithm to (\ref{Breg_iter}) in $l^\text{th}$ step (similar to Algorithm \ref{alg:primal-dual}), we obtain an approximate minimizer $(\theta^{l+1},\rho^{l+1},\vv^{l+1})$ and the corresponding dual variable $\{\psi_i^{l+1}\}_{i=1}^r$. Then we update $\vq^{l}$ as a subgradient of $J(\cdot)$ at $\theta^{l+1}$. According to KKT condition,
$$0\in\partial J(\theta^{l+1})-\vq^l+\sum_{i=1}^r \frac{\partial}{\partial\theta}c_i(\rho^{l+1},\vv^{l+1},\theta^{l+1})\psi_i^{l+1}.$$
So we take $\vq^{l+1}=\vq^l-\sum_{i=1}^{r}\frac{\delta}{\delta\theta}
c_i(\rho^{l+1},\vv^{l+1},\theta^{l+1})\psi_i^{l+1}$ to have $\vq^{l+1}\in\partial_{\theta}J(\theta^{l+1})$. We run several Bregman iterations until the numerical result converges.
\begin{algorithm}
	\caption{Bregman iteration}
	\label{alg:Bregman}
	\begin{algorithmic}
		\STATE{\textbf{Input:} the observations $\hat{\rho},\hat{\vv}$, all the scaling parameters $\alpha,\alpha_0,\beta$, the norm index $p$, and other prior known parameters in the inverse model. Input the iteration parameters $\tau_\rho,\tau_\vv,\tau_\theta,\sigma,\delta$.}
		\STATE
		{\textbf{Initialization:} Set $\rho$ as some suitable constant $>0$, $\vv=0$, $\theta$ to be a positive constant. Set $\vq^0=0$.}
		\FOR{l=0,1,2,\ldots (Until converges)}
		\STATE{Solve (\ref{Breg_iter}) with primal-dual algorithm, get $(\theta^{l+1},\rho^{l+1},\vv^{l+1})$, and the dual variable $\{\psi_i^{l+1}\}_{i=1}^r$ }
		\STATE{Update $\vq^{l+1}=\vq^l-\sum_{i=1}^{r}\frac{\partial}{\partial\theta}
			c_i(\rho^{l+1},\vv^{l+1},\theta^{l+1})\psi_i^{l+1}$.}
		\ENDFOR
	\end{algorithmic}
\end{algorithm}

\section{Computational Results}
\label{chap:comp_result}
We have tested both Model \ref{inv-ground-metric} and Model \ref{inv-MMD} with Algorithm \ref{alg:primal-dual}. We used the numerical result of the forward problem (\ref{MFG-discrete}) as the observations in our inverse-problem tests. This was to recreate the corresponding parameters, the metric kernel $g_0$ and the convolution kernel $\tilde{K}$ from the full observations of $\hat{\rho},\hat{\vv}$.

The scaling parameters in the model, in several tests, were taken as
\begin{equation}
\label{standard-para}
\alpha=\alpha_0=\frac{1}{\|\hat{\rho}\|^2},\quad\beta=\frac{1}{\hat{\vv}^T\hat{\vv}}.
\end{equation}
In other words, the unknown variables density measure and the velocity field in the models were scaled to the same size according to their observations. It was hard to pre-select  $\gamma$. Thus in the test, we took $\gamma$ in a wide range. The parameter $p$ was determined according to the properties of the ground metric and the convolution kernel. If the regularized variable was smooth, $p=2$ was used. When the target variable was sparse, $p=1$ was chosen.

In the computation test, we made the assumption that some minor information about the parameters $g_0$, $\tilde{K}$ was known beforehand to reduce the uncertainty during the iterations. The details are covered below with each example.

In our test for Model \ref{disc-inv-ground-metric} to learn the metric kernel $g_0$, we took $F(\rho)=\rho^2/2$ in all of our tests. In the test for Model \ref{disc-inv-kernel}, the real kernel $\tilde{K}$ was always taken as the standard form
$\tilde{K}(x)=\exp(-\vx^TA\vx/\epsilon) $
as in Example \ref{e.x.:MMD}.

The observations $\hat{\rho},\hat{\vv}$ were noised to different extents by i.i.d. additive noise on each pixel. The noise subjects to a uniform distribution and relies on the norm of $\hat{\rho},\hat{\vv}$:
\begin{equation}
\label{noise}
\begin{split}
&\epsilon_\rho\sim\epsilon^*\|\hat{\rho}\|\cdot U[-0.5,0.5]\quad i.i.d.\\ &\epsilon_{\vv}\sim\epsilon^*\|\hat{\vv}\|\cdot U[-0.5,0.5]^d\quad i.i.d.,
\end{split}
\end{equation}
where $\epsilon^*>0$ is the noise factor, $U[-0.5,0.5]$ is the uniform distribution between $[-0.5,0.5]$, and $U[-0.5,0.5]^d$ is the uniform distribution in $[-0.5,0.5]^d$. Our codes were written in MATLAB.
\begin{test}
	\label{test:2d-metric}
	\rm We tested Algorithm \ref{alg:primal-dual} in 2 dimensions. No noise was impacted on the observed data. We took $G_M$ as the linear form of $g_0$:
	\begin{equation*}
	G(x,y)=\left(
	\begin{array}{cc}
	g_0+4 & g_0+2\\
	g_0+2 & 2g_0+1
	\end{array}
	\right).
	\end{equation*}
	We discretized the problem on the $50\times50 \times30$ grid.
	The square space was uniformly scattered into $50\times50$ boxes while the time $[0,T]$ was scattered into 30 elements. Beforehand, we had information about $g_0$ in a single pile of pixels. The values of $g_0$ at boxes $\{C(1,\cdot)\}$ were obtained. For the scaling parameter, we took $\alpha=10^3/\|\hat{\rho}\|^2,\beta=1/\|\hat{\vv}\|^2,\alpha_0=0,\gamma=10^{-2}$. Supposed that the real metric kernel was smooth, we took $p=2$. The iteration step sizes for the primal variables were $\tau_\rho=\tau_\vv=10^{-6},\tau_g=10^{-4}$, and the step size for dual variables was $\sigma=10^{-6}$. We let it run for $2\times 10^5$ iterations and observed the variables have converged. The result is displayed in Figure \ref{result-patern-2-11}.
	\begin{figure}[htbp]
		
		\centering
		\includegraphics[width=0.9\textwidth]{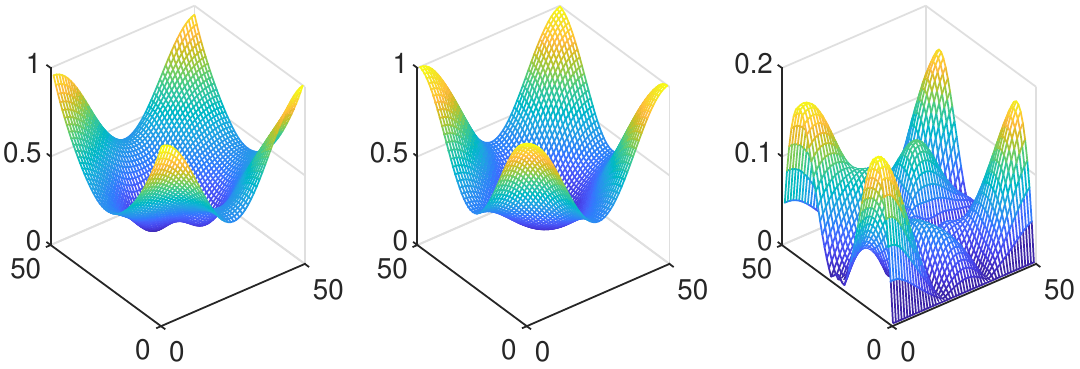}
		\caption{Ground metric recreated in 2-dimensional basis under $H^2$ regularization. From left to right: $g_0$ learned from the observations, real $\bar{g_0}$, the absolute difference between the learned $g_0$ and the real data $|g_0-\bar{g_0}|$.}
		\label{result-patern-2-11}
	\end{figure}
\end{test}

\begin{test}
	\rm
	We tested the 2-dimensional based inverse model for the convolution kernel.
	
	In this example, we tested Algorithm \ref{alg:primal-dual} on 2-dimensional based Model \ref{disc-inv-kernel}. No noise was impacted on the observations. The real kernel for the forward problem was taken as:
	\begin{equation*}
	K(\vx)=\exp\left(-\vx^T
	\left(
	\begin{array}{cc}
	3&1\\
	1&3
	\end{array}
	\right)\vx/0.5\right).
	\end{equation*}
	
	We discretized the problem on the $24\times24 \times30$ grid. We had information about $\tilde{K}$ in a single pile of pixels. For the scaling parameter, we took $\alpha=100/\|\hat{\rho}\|^2,\beta=1/\|\hat{\vv}\|^2,\alpha_0=0,\gamma=10^{-3}$. Assumed that we knew the kernel being smooth as a priori, we took $p=2$. The iteration step sizes for the primal variables were $\tau_\rho=\tau_\vv=10^{-5},\tau_g=10^{-3}$, and the step size for dual variables was $\sigma=10^{-5}$. We took $1.5\times10^5$ iterations. In observation, the variables had converged after these iterations. The result is shown in Figure \ref{kernel-2d-3-26}.
	
	\begin{figure}[htbp]
		
		\centering
		\includegraphics[width=0.9\textwidth]{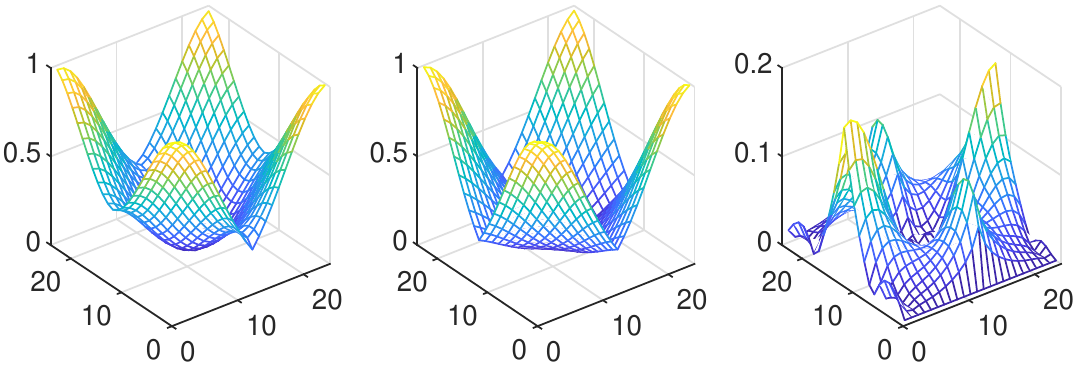}
		\caption{Convolution kernel recreated in 2-dimensional basis under $H^2$ regularization. To have an overview of the kernel on $\mathbb{T}^2$. We recreated the $K(x,0)$ on $[0,1]^2$ with the learned $\tilde{K}$ on $[0,1/2]^2$. From left to right: kernel $K(x,0)$ learned from the observations, real kernel $K_{real}(x,0)$, the absolute difference between the learned parameter and the real one $|K(x,0)-K_{real}(x,0)|$.}
		\label{kernel-2d-3-26}
	\end{figure}
\end{test}

\begin{test}
	\label{test:noise-metric}
	\rm Then we investigated the case where noise was impacted on our observation. We tested examples in 1-dimensional basis. Here $G_M=g_0$ was taken. We discretized the problem on a $50\times 30$ grid in $x,t$, with space uniformly spaced in each variable.
	The noise $\epsilon_\rho,\epsilon_{\vv}$ were taken as in (\ref{noise}). Noise factor $\epsilon^*$ adopted varied in $\{0.1,0.4,1\}$. We assume that we have the noised data. We show noised $\hat{\rho},\hat{\vv}$ in Figure \ref{noise-rho-v}.
	Assumed that we had information about $G_M$ at a single pixel, and the value of $G_M$ at that pixel was fixed in the iteration.
	The scaling parameter $\alpha=1/\|\hat{\rho}\|^2,\beta=1/\|\hat{v}\|^2,\alpha_0=0$, and $\gamma$ varied in $\{10^{-8},10^{-7},\ldots,10^{-3}\}$. The real ground metric was supposed to be smooth, thus $p=2$ was taken. The iteration step size for all primal variables was $\tau=2\times 10^{-3}$, and the step size for dual variables was $\sigma=10^{-3}$. We took iterations of $6\times10^4$ times. In observation, the iteration had converged in such a setting.
	The results are shown in Figure \ref{noise-metric}.
	
	\begin{figure}[htbp]
		
		\centering
		\includegraphics[width=0.9\textwidth]{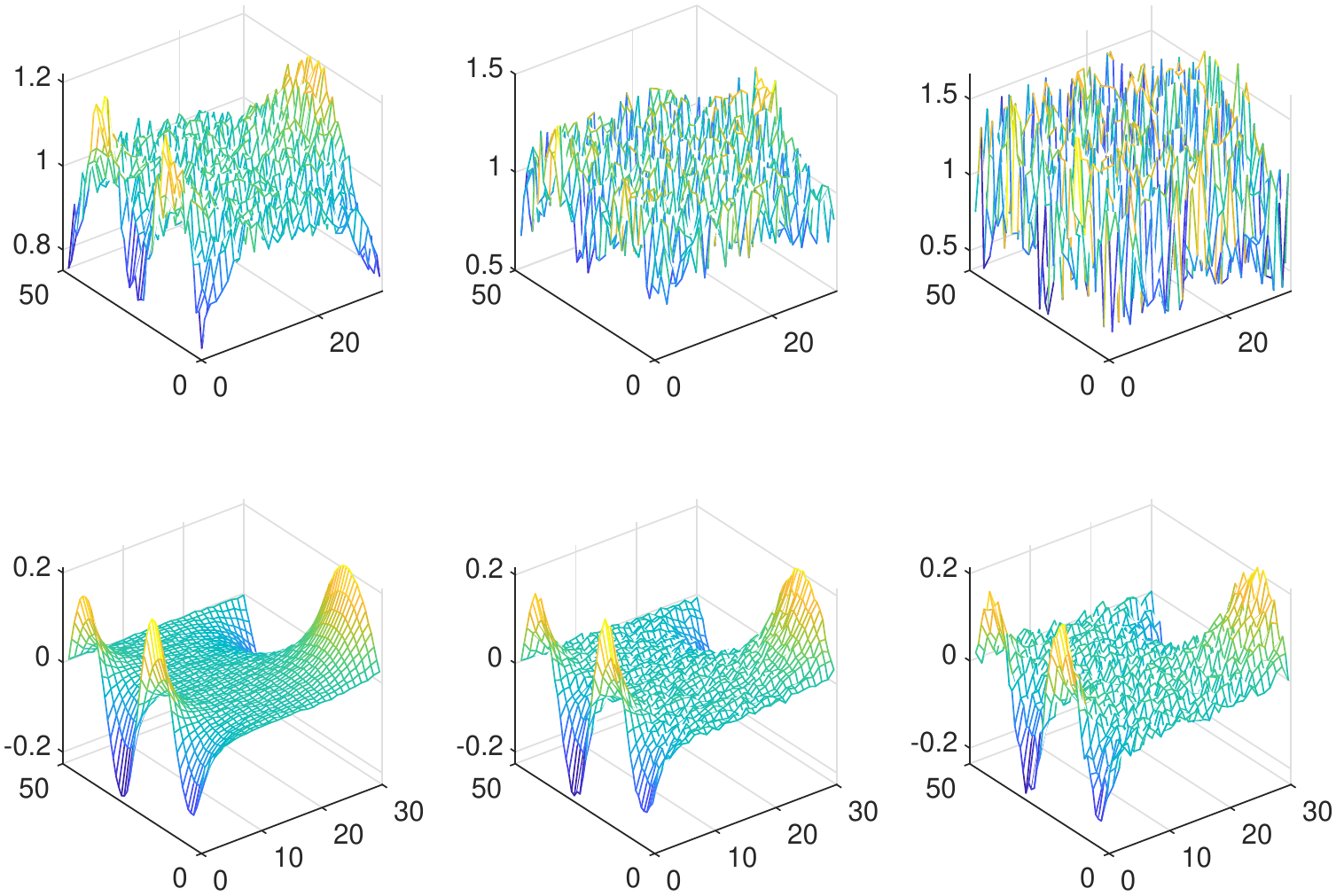}
		\caption{The noised $\hat{\rho},\hat{\vv}$ with different noise factor. The first row is for noised $\hat{\rho}$ with $\epsilon^*=0.1,0.4,1$. The second row is for noised $\hat{\vv}$ with $\epsilon^*=0.1,0.4,1$.}
		\label{noise-rho-v}
	\end{figure}
	
	%
	%
	%
	%
	%
	
	\begin{figure}[htbp]
		\centering
		\includegraphics[width=0.9\textwidth]{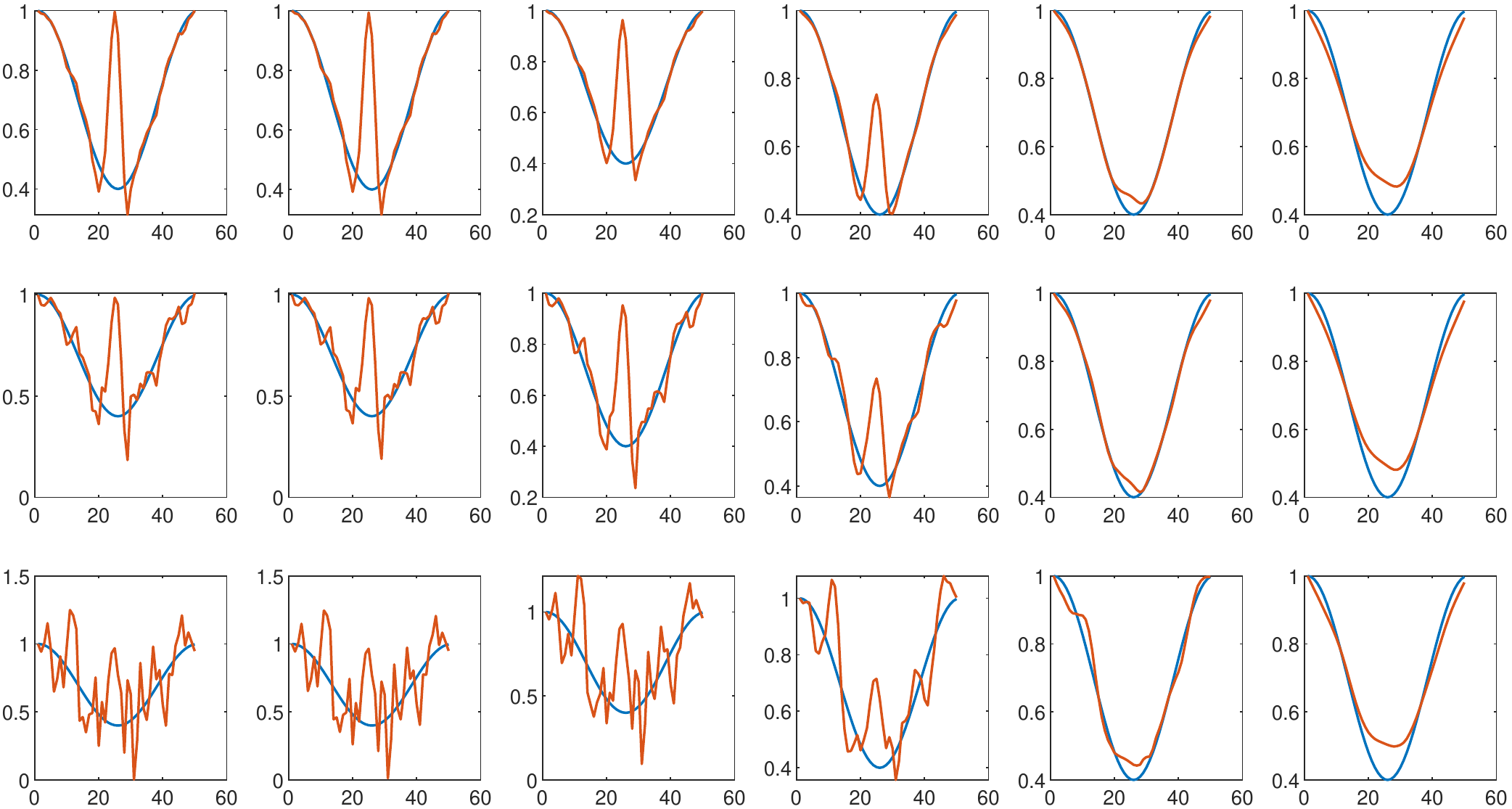}
		\caption{Recreated ground metric $G_M$ when the observations were corrupted with noise. In each sub-figure, the red curve is the learned ground metric while the blue curve is the real metric. From up to down, noise factor $\epsilon^*=0.1,0.4,1$. From left to right, $\gamma$ varies from $10^{-8}$ to $10^{-3}$.}
		\label{noise-metric}
	\end{figure}
	
\end{test}

\begin{test}
	\label{test:kernel}
	\rm We also tested the robustness of Model \ref{disc-inv-kernel}. We tested the example in the 1-dimensional basis where $K(x)=\exp(-x^2/0.1)$. The space-time was discretized into a $50\times 30$ grid. The only information about $K$ was that $K(0,0)=1$. And in the iteration, we fixed the value of $\tilde{K}(0)$ to be 1.
	The scaling parameter $\alpha=1/\|\hat{\rho}\|^2,\beta=1/\|\hat{v}\|^2,\alpha_0=0$. $\gamma$ varied in $\{10^{-6},10^{-5},10^{-4},10^{-3}\}$. Due to the quadratic exponential format, the kernel was considered smooth, thus $p=2$ was taken. The iteration step size for all primal variables was $\tau=10^{-3}$, and the step size for dual variables was $\sigma=10^{-3}$ as well. Still, the noise factor $\epsilon^*$ varied in $\{0.1,0.4,1\}$.
	We took $3\times 10^6$ iterations. The variables fully converged in our observation. The computational results are compiled in Figure \ref{noise-kernel}.

	\begin{figure}[htbp]
		
		\centering
		\includegraphics[width=0.9\textwidth]{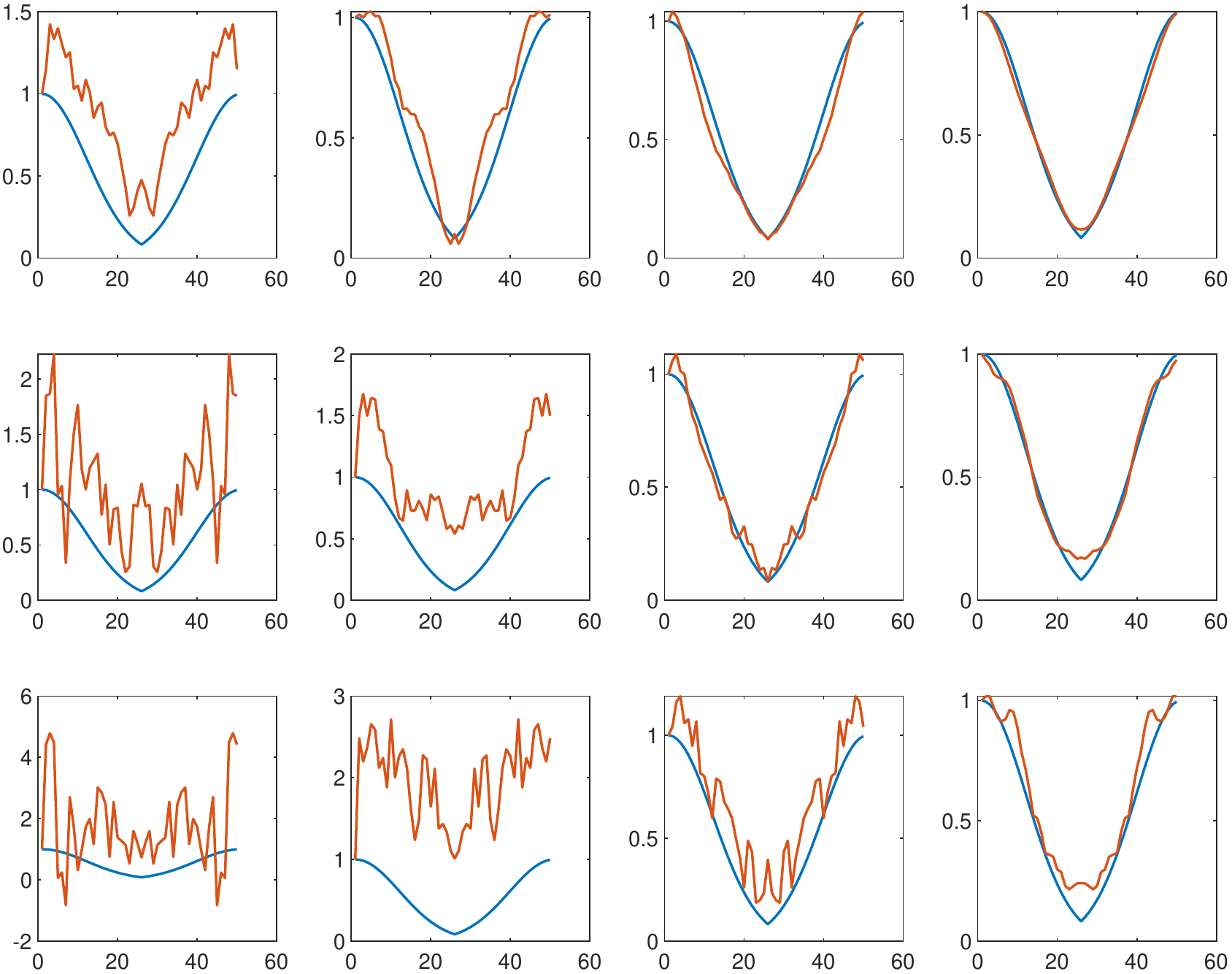}
		\caption{Recreated convolution kernel $K(x,0)$ when the observations were corrupted with noise.
			From up to down, noise factor $\epsilon^*=0.1,0.4,1$. From left to right, $\gamma$ varies from $10^{-6}$ to $10^{-3}$.}
		\label{noise-kernel}
	\end{figure}

\end{test}

\begin{test}
	\label{test:Breg_metric}
	\rm In this test, we tested the efficiency of Algorithm \ref{alg:Bregman}. We worked on the same data as Test \ref{test:noise-metric}. The observations $\hat{\rho},\hat{\vv}$ were corrupted by noise at a level $\epsilon^*=1$. We took $\gamma=10^{-2}$, and the first 7 Bregman iterations are shown in Figure \ref{Breg_metric_-2_1}.
	\begin{figure}[htbp]
		\centering
		\includegraphics[width=0.9\textwidth]{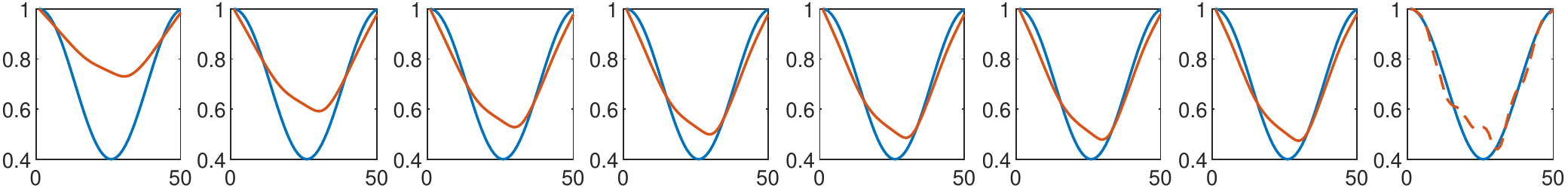}
		\caption{The recreated metric kernel from noised data in first 7 Bregman iterations. $\gamma=10^{-2}$.
			(The last figure is for recreated ground metric from Algorithm \ref{alg:primal-dual} when taking  $\gamma=10^{-4}$ as an appropriate optimal value.)}
		\label{Breg_metric_-2_1}
	\end{figure}
\end{test}

\begin{remark}
	\rm In Test \ref{test:Breg_metric}, by applying Bregman iteration, starting with a non-optimal $\gamma$, we had better result for ground metric. The recreated ground metric even converged to the one recreated with the optimal $\gamma$.
\end{remark}
Several more numerical results are provided in Appendix \ref{appen:comp}.

\section{Discussion}
In this paper, we introduced some inverse-problem models for MFGs, which are PDE-constrained variational problems. 

MFGs involve PDE systems, including the Fokker-Planck equations and HJE, which arise naturally from physical systems, such as Schr{\"o}dinger equations, Schrodinger bridge problems, and optimal control problems in finance, robotics path planning and game theory. These problems are mean-field descriptions of classical observations in finite agents' dynamics. One can view our model as a mean-field generalization of classic (finite-agent) inverse problems. It is an extension of classical observations to density space. Mean-field limit analyses and the well-posedness of the model are challenging future directions in this uncultivated area.

The MFG PDE system involves nonlinear Hamiltonian  constraints in the sample space. Our method is only a first trial towards its computation, based on the classical PDE method. The study of convexity properties, parallel computation, and designing fast and efficient algorithms for the proposed models are also interesting future directions.

From the lens of information science and machine learning, our model fits the goal of learning Hamiltonians and physics from data and observations. Here we have concretely modeled the physics by a Hamiltonian on the underlying sample space and design the cost functional in the density space to fit the model. To treat the density variational problem in our model, many information sciences and machine learning approaches can be considered in the future. For example, we can further apply the other perspectives of density using probability models in machine learning and transport information geometry~\cite{li_transport_2020,li_hessian_2020}. Typical models include Gaussian families, generative models, and reinforcement learning~\cite{gu_dynamic_2020,guo_general_2020}. We leave detailed studies of these models for inverse problems in the near future.


\newpage

\appendix
\addappheadtotoc

\section{Details of the algorithm in Subsection \ref{chap:disc-alg}}
\label{apped:alg}
In this section, we present the iteration steps of Algorithm \ref{alg:primal-dual} in Subsection \ref{chap:disc-alg}.

The first algorithm is for Model \ref{disc-inv-ground-metric}, where $\mathcal{F}$ is the integral of a convex function $F$. To solve it, we introduce the dual variables for the constraints of (\ref{MFG-inv-disc1}): $(\psi_{i+\frac{e_1}{2},j},\psi_{i+\frac{e_2}{2},j}), i\in V,j=2,3,\ldots,n$, $\Phi_{i,j},i\in V,j=1,2,\ldots,n$, $\chi_{i,j},i\in V,j=1,2,\ldots,n$, and $\Theta_{i,j}=\left(\Theta_{x,i_2,j},\Theta_{y,i_1,j}\right)^T,i\in V,j=1,2,\ldots,n$. Summing the objective function and Lagrangian multipliers, we obtain
\begin{equation*}
\begin{split}
&\operatorname*{min}_{g_0,\rho,v}
\operatorname*{max}_{\psi,\Phi,\chi,\Theta}L(g_0,\rho,v;\psi,\Phi,\chi,\Theta):=\\
&\operatorname*{min}_{g_0,\rho,v}
\operatorname*{max}_{\psi,\Phi,\chi,\Theta}
\sum_{i\in V}\sum_{j=1}^{n+1}\frac{\alpha}{2}\left(\rho_{i,j-\frac{1}{2}}-\hat{\rho}_{i,j-\frac{1}{2}}\right)^2+
\sum_{i\in V}\sum_{j=1}^n\frac{\beta}{2}\|(v_{i+\frac{e_v}{2},j})_v-(\hat{v}_{i+\frac{e_v}{2},j})_v\|^2
\\
&\quad
\begin{split}
&+\sum_{i\in V}
\frac{\alpha_0}{2\triangle t}
\left(
\left(\rho_{i,\frac{1}{2}}-\hat{\rho}_{i,\frac{1}{2}}\right)^2
+
\left(\rho_{i,n+\frac{1}{2}}-\hat{\rho}_{i,n+\frac{1}{2}}\right)^2
\right)
+\sum_{i\in V}\sum_{e_v}\frac{\gamma}{p\triangle t}
\left|\frac{g_{0,i+e_v}-g_{0,i}}{\triangle x}\right|^p
\\
&+\sum_{i\in V}\sum_{j=2}^n
(\psi_{i+\frac{e_v}{2},j})_v\cdot\left(
\left(\frac{\xi_{i+e_v,j-\frac{1}{2}}-\xi_{i,j-\frac{1}{2}}}{\triangle x}\right)_v+\left(\frac{(w_{i+\frac{e_v}{2},j})_v-(w_{i+\frac{e_v}{2},j-1})_v}{\triangle t}\right)
\right)\\
&+\sum_{i\in V}\sum_{j=1}^n
\Phi_{i,j}
\left(
\left(\rho_{i,j+\frac{1}{2}}-\rho_{i,j-\frac{1}{2}}\right)/\triangle t+\sum_{e_v}\left(m_{i+\frac{e_v}{2},j}-m_{i-\frac{e_v}{2},j}\right)/\triangle x
\right)\\
&+\sum_{i\in V}\sum_{j=1}^n
\chi_{i,j}
\left(
\frac{w_{i+\frac{e_1}{2}+e_2,j}-w_{i+\frac{e_1}{2},j}}{\triangle x}-\frac{w_{i+\frac{e_2}{2}+e_1,j}-w_{i+\frac{e_2}{2},j}}{\triangle x}
\right)\\
&+
\sum_{i\in V}\sum_{j=1}^n\Theta_{x,i_2,j}w_{i+\frac{e_1}{2},j}\frac{1}{\triangle x}
+
\sum_{i\in V}\sum_{j=1}^n\Theta_{y,i_1,j}w_{i+\frac{e_2}{2},j}\frac{1}{\triangle x},
\end{split}
\end{split}
\end{equation*}
where
\begin{equation*}
\left\{
\begin{split}
&\xi_{i,j-\frac{1}{2}}=
\frac{1}{2}(v_{i+\frac{e_v}{2},j})_v^TG_{M,i}(v_{i+\frac{e_v}{2},j})_v
-
F'(\rho_{i,j-\frac{1}{2}}),\quad i\in V,j=2,3,\ldots,n\\
&
(w_{i+\frac{e_v}{2},j})_v=G_{M,i}(v_{i+\frac{e_v}{2},j})_v,\quad i\in V,j=1,2,\ldots,n\\
&
(m_{i+\frac{e_v}{2},j})_v=\rho_{i,j-\frac{1}{2}}(v_{i+\frac{e_v}{2},j})_v,\quad i\in V,j=1,2,\ldots,n
\end{split}
\right.
\end{equation*}
By changing the order of sums, we further obtain:
\begin{align*}
\begin{split}
&\sum_{i\in V}\sum_{j=2}^n
(\psi_{i+\frac{e_v}{2},j})_v\cdot\left(
\left(\frac{\xi_{i+e_v,j-\frac{1}{2}}-\xi_{i,j-\frac{1}{2}}}{\triangle x}\right)_v+\left(\frac{(w_{i+\frac{e_v}{2},j})_v-(w_{i+\frac{e_v}{2},j-1})_v}{\triangle t}\right)
\right)\\
=&\sum_{i\in V}\sum_{j=2}^n
\sum_{e_v}\frac{\psi_{i-\frac{e_v}{2},j}-\psi_{i+\frac{e_v}{2},j}}{\triangle x}\xi_{i,j-\frac{1}{2}}
+
\sum_{i\in V}\sum_{j=2}^{n-1}
\frac{(\psi_{i+\frac{e_v}{2},j})_v-(\psi_{i+\frac{e_v}{2},j+1})_v}{\triangle t}\cdot(w_{i+\frac{e_v}{2},j})_v
\\
&\quad
+
\sum_{i\in V}\frac{(\psi_{i+\frac{e_v}{2},n})_v}{\triangle t}\cdot(w_{i+\frac{e_v}{2},n})_v
-\sum_{i\in V}\frac{(\psi_{i+\frac{e_v}{2},2})_v}{\triangle t}\cdot(w_{i+\frac{e_v}{2},1})_v,
\end{split}\\
\begin{split}
&\sum_{i\in V}\sum_{j=1}^n
\Phi_{i,j}
\left(
\left(\rho_{i,j+\frac{1}{2}}-\rho_{i,j-\frac{1}{2}}\right)/\triangle t+\sum_{e_v}\left(m_{i+\frac{e_v}{2},j}-m_{i-\frac{e_v}{2},j}\right)/\triangle x
\right)\\
=&\sum_{i\in V}\sum_{j=2}^n
\frac{\Phi_{i,j-1}-\Phi_{i,j}}{\triangle t}\rho_{i,j-\frac{1}{2}}
+
\sum_{i\in V}\frac{\Phi_{i,n}}{\triangle t}\rho_{i,n+\frac{1}{2}}
-
\sum_{i\in V}\frac{\Phi_{i,1}}{\triangle t}\rho_{i,\frac{1}{2}}
+\sum_{i\in V}\sum_{j=1}^n
\left(
\frac{\Phi_{i,j}-\Phi_{i+e_v,j}}{\triangle x}
\right)_v^T(m_{i+\frac{e_v}{2},j})_v,
\end{split}\\
\begin{split}
&\sum_{i\in V}\sum_{j=1}^n
\chi_{i,j}
\left(
\frac{w_{i+\frac{e_1}{2}+e_2,j}-w_{i+\frac{e_1}{2},j}}{\triangle x}-\frac{w_{i+\frac{e_2}{2}+e_1,j}-w_{i+\frac{e_2}{2},j}}{\triangle x}
\right)\\
=&\sum_{i\in V}\sum_{j=1}^n\frac{\chi_{i-e_2,j}-\chi_{i,j}}{\triangle x}w_{i+\frac{e_1}{2},j}
-
\frac{\chi_{i-e_1,j}-\chi_{i,j}}{\triangle x}w_{i+\frac{e_2}{2},j}.
\end{split}
\end{align*}
It is more convenient to optimize using the primal variables.

In the primal step, we apply gradient descent to the Lagrangian, $L$, with respect to the primal variables, leading to
\begin{equation*}
\left\{
\begin{split}
&\rho_{i,j-\frac{1}{2}}^{k+1}=
\rho_{i,j-\frac{1}{2}}^{k}-
\tau_\rho \frac{\partial}{\partial\rho_{i,j-\frac{1}{2}}}L(g_0^k,\rho^k,v^k;\psi^k,\Phi^k,\chi^k,\Theta^k)\\
&(v_{i+\frac{e_v}{2},j}^{k+1})_v=
(v_{i+\frac{e_v}{2},j}^k)_v-
\tau_\vv\frac{\partial}{\partial(v_{i+\frac{e_v}{2},j})_v}L(g_0^k,\rho^k,v^k;\psi^k,\Phi^k,\chi^k,\Theta^k)
\\
&g_{0,i}^{k+1}=g_{0,i}^{k}
-\tau_{g}\frac{\partial}{\partial g_{0,i}}L(g_0^k,\rho^k,v^k;\psi^k,\Phi^k,\chi^k,\Theta^k),
\end{split}
\right.
\end{equation*}
where $\tau_\rho,\tau_\vv,\tau_g$ are step sizes for $\rho,\vv,g_0$.  We note that when $p=1$, the gradient of $|g_0|$ does not exist at 0, so we take the sub-gradient.
In computation, we apply a different descent step size for each of the primal variables to overcome the multi-scaling issue, which accelerates the iteration rate. We update the dual
variables by:
\begin{equation*}
\left\{
\begin{split}
&(\psi_{i+\frac{e_v}{2},j}^{k+1})_v=(\psi_{i+\frac{e_v}{2},j}^k)_v+
\sigma\left(
\left(\frac{\xi^*_{i+e_v,j-\frac{1}{2}}-\xi^*_{i,j-\frac{1}{2}}}{\triangle x}\right)_v
+
\left(\frac{(w_{i+\frac{e_v}{2},j}^*)_v-(w_{i+\frac{e_v}{2},j-1}^*)_v}{\triangle t}\right)
\right),\\
&\qquad
i\in V,j=2,3,\ldots,n\\
&\Phi_{i,j}^{k+1}=\Phi_{i,j}^{k}+\sigma\left(\left(\rho^*_{i,j+\frac{1}{2}}-\rho^*_{i,j-\frac{1}{2}}\right)/\triangle t+\sum_{e_v}\left(m^*_{i+\frac{e_v}{2},j}-m^*_{i-\frac{e_v}{2},j}\right)/\triangle x\right),\\
&\qquad i\in V,j=1,2,\ldots,n\\
&\chi_{i,j}^{k+1}=\chi_{i,j}^{k}+\sigma\left(
\frac{w^*_{i+\frac{e_1}{2}+e_2,j}-w^*_{i+\frac{e_1}{2},j}}{\triangle x}-\frac{w^*_{i+\frac{e_2}{2}+e_1,j}-w^*_{i+\frac{e_2}{2},j}}{\triangle x}
\right),\,i\in V,j=1,2,\ldots,n\\
&
\Theta_{x,i_2,j}^{k+1}=\Theta_{x,i_2,j}^{k}+\sigma\left(\sum_{i_1}w^*_{i+\frac{e_1}{2},j}\right),\quad i_2=1,2,\ldots,m,j=1,2\ldots,n\\
&\Theta_{y,i_1,j}^{k+1}=\Theta_{y,i_1,j}^{k}+\sigma
\left(
\sum_{i_2}w^*_{i+\frac{e_2}{2},j}
\right),\quad i_1=1,2,\ldots,m,j=1,2\ldots,n,
\end{split}
\right.
\end{equation*}
where
\begin{equation*}
\left\{
\begin{split}
&\rho_{i,j-\frac{1}{2}}^*=
2\rho_{i,j-\frac{1}{2}}^{k+1}-\rho_{i,j-\frac{1}{2}}^k
\quad i\in V,j=1,2,\ldots,n+1\\
&(v_{i+\frac{e_v}{2},j}^*)_v=
2(v_{i+\frac{e_v}{2},j}^{k+1})_v-(v_{i+\frac{e_v}{2},j}^k)_v\quad i\in V,j=1,2\ldots,n\\
&g_{0,i}^*=2g_{0,i}^{k+1}-g_{0,i}^k\quad i\in V\\
&G_{M,i}^*=
\left(
\begin{array}{cc}
f_{11}({g_{0,i}^*})&
f_{12}({g_{0,i}^*})\\
f_{21}({g_{0,i}^*})&
f_{22}({g_{0,i}^*})
\end{array}
\right)
\quad i\in V\\
&\xi^*_{i,j-\frac{1}{2}}=
\frac{1}{2}{(v_{i+\frac{e_v}{2},j}^*)_v}^TG^*_{M,i}(v_{i+\frac{e_v}{2},j}^*)_v
-
F'(\rho^*_{i,j-\frac{1}{2}})\quad i\in V,j=2,3,\ldots,n\\
&
(w_{i+\frac{e_v}{2},j}^*)_v=G^*_{M,i}(v_{i+\frac{e_v}{2},j}^*)_v\quad i\in V,j=1,2,\ldots,n\\
&
(m_{i+\frac{e_v}{2},j}^*)_v=\rho^*_{i,j-\frac{1}{2}}(v_{i+\frac{e_v}{2},j}^*)_v\quad i\in V,j=1,2,\ldots,n
\end{split}
\right.
\end{equation*}

The partial difference of Lagrangian with respect to $\rho,\vv,g_0$ can be explicitly written as:
\begin{equation*}
\small{
\begin{split}
&\frac{\partial}{\partial\rho_{i,j-\frac{1}{2}}}L(g_0^k,\rho^k,v^k;\psi^k,\Phi^k,\chi^k,\Theta^k)=\\
&\left\{
\begin{split}
&\left(
\alpha+\frac{\alpha_0}{\triangle t}
\right)\left(\rho^k_{i,\frac{1}{2}}-
\hat{\rho}_{i,\frac{1}{2}}\right)
-\frac{\Phi^k_{i,1}}{\triangle t}+
\left(
\frac{\Phi^k_{i,1}-\Phi^k_{i+e_v,1}}{\triangle x}
\right)_v\cdot
(v_{i+\frac{e_v}{2},1}^k)_v
,\quad j=1\\
&\alpha\left(
\rho^k_{i,j-\frac{1}{2}}-\hat{\rho}_{i,j-\frac{1}{2}}
\right)
+
\frac{\Phi^k_{i,j-1}-\Phi^k_{i,j}}{\triangle t}
+
\left(
\frac{\Phi^k_{i,j}-\Phi^k_{i+e_v,j}}{\triangle x}
\right)_v\cdot(v_{i+\frac{e_v}{2},j}^k)_v
\\
&\qquad+\sum_{e_v}\frac{\psi^k_{i+\frac{e_v}{2},j}-\psi^k_{i-\frac{e_v}{2},j}}{\triangle x}F''(\rho^k_{i,j-\frac{1}{2}}),\quad
j=2,3,\ldots,n\\
&\left(\alpha+\frac{\alpha_0}{\triangle t}
\right)\left(\rho^k_{i,n+\frac{1}{2}}-
\hat{\rho}_{i,n+\frac{1}{2}}\right)+\frac{\Phi^k_{i,n}}{\triangle t},\quad j=n+1,
\end{split}
\right.
\end{split}
}
\end{equation*}

\begin{equation*}
\small
{
\begin{split}
&\frac{\partial}{\partial(v_{i+\frac{e_v}{2},j})_v}L(g_0^k,\rho^k,v^k;\psi^k,\Phi^k,\chi^k,\Theta^k)=\\
&\left\{
\begin{split}
&\beta\left(
(v_{i+\frac{e_v}{2},1}^k)_v-(\hat{v}_{i+\frac{e_v}{2},1})_v
\right)
+
\rho^k_{i,\frac{1}{2}}
\left(
\frac{\Phi^k_{i,1}-\Phi^k_{i+e_v,1}}{\triangle x}
\right)_v
-
G_{M,i}^k\frac{(\psi_{i+\frac{e_v}{2},2}^k)_v}{\triangle t}
\\
&
\qquad
+
\frac{\chi^k_{i-e_2,1}-\chi^k_{i,1}}{\triangle x}
\left(
\begin{array}{c}
f_{11}(g_{0,i}^k)\\
f_{12}(g_{0,i}^k)
\end{array}
\right)
-
\frac{\chi^k_{i-e_1,1}-\chi^k_{i,1}}{\triangle x}
\left(
\begin{array}{c}
f_{12}(g_{0,i}^k)\\
f_{22}(g_{0,i}^k)
\end{array}
\right)
+
\frac{\Theta^k_{x,i_2,1}}{\triangle x}
\left(
\begin{array}{c}
f_{11}(g_{0,i}^k)\\
f_{12}(g_{0,i}^k)
\end{array}
\right)
\\
&
\qquad
+
\frac{\Theta^k_{y,i_1,1}}{\triangle x}
\left(
\begin{array}{c}
f_{12}(g_{0,i}^k)\\
f_{22}(g_{0,i}^k)
\end{array}
\right)
,\quad j=1\\
&\beta\left(
(v_{i+\frac{e_v}{2},j}^k)_v-(\hat{v}_{i+\frac{e_v}{2},j})_v
\right)
+
\rho^k_{i,j-\frac{1}{2}}
\left(
\frac{\Phi^k_{i,j}-\Phi^k_{i+e_v,j}}{\triangle x}
\right)_v
\\
&\qquad
+
\sum_{e_v}
\left(
\frac{\psi^k_{i-\frac{e_v}{2},j}-\psi^k_{i+\frac{e_v}{2},j}}{\triangle x}
\right)
G_{M,i}^k(v_{i+\frac{e_v}{2},j}^k)_v+
G_{M,i}^k
\left(
\frac{\psi^k_{i+\frac{e_v}{2},j}-\psi^k_{i+\frac{e_v}{2},j+1}}{\triangle t}
\right)_v
\\
&\qquad
+
\frac{\chi^k_{i-e_2,j}-\chi^k_{i,j}}{\triangle x}
\left(
\begin{array}{c}
f_{11}(g_{0,i}^k)\\
f_{12}(g_{0,i}^k)
\end{array}
\right)
-
\frac{\chi^k_{i-e_1,j}-\chi^k_{i,j}}{\triangle x}
\left(
\begin{array}{c}
f_{12}(g_{0,i}^k)\\
f_{22}(g_{0,i}^k)
\end{array}
\right)
+
\frac{\Theta^k_{x,i_2,j}}{\triangle x}
\left(
\begin{array}{c}
f_{11}(g_{0,i}^k)\\
f_{12}(g_{0,i}^k)
\end{array}
\right)\\
&
\qquad
+
\frac{\Theta^k_{y,i_1,j}}{\triangle x}
\left(
\begin{array}{c}
f_{12}(g_{0,i}^k)\\
f_{22}(g_{0,i}^k)
\end{array}
\right)
,\quad j=2,3,\ldots,n-1\\
&\beta\left(
(v_{i+\frac{e_v}{2},n}^k)_v-(\hat{v}_{i+\frac{e_v}{2},n})_v
\right)
+
\rho^k_{i,n-\frac{1}{2}}
\left(
\frac{\Phi^k_{i,n}-\Phi^k_{i+e_v,n}}{\triangle x}
\right)_v
\\
&
\qquad
+
\sum_{e_v}
\left(
\frac{\psi^k_{i-\frac{e_v}{2},n}-\psi^k_{i+\frac{e_v}{2},n}}{\triangle x}
\right)
G_{M,i}^k(v_{i+\frac{e_v}{2},n}^k)_v+
G_{M,i}^k\frac{(\psi_{i+\frac{e_v}{2},n}^k)_v}{\triangle t}
\\
&
\qquad
+
\frac{\chi^k_{i-e_2,n}-\chi^k_{i,n}}{\triangle x}
\left(
\begin{array}{c}
f_{11}(g_{0,i}^k)\\
f_{12}(g_{0,i}^k)
\end{array}
\right)
-
\frac{\chi^k_{i-e_1,n}-\chi^k_{i,n}}{\triangle x}
\left(
\begin{array}{c}
f_{12}(g_{0,i}^k)\\
f_{22}(g_{0,i}^k)
\end{array}
\right)
+
\frac{\Theta^k_{x,i_2,n}}{\triangle x}
\left(
\begin{array}{c}
f_{11}(g_{0,i}^k)\\
f_{12}(g_{0,i}^k)
\end{array}
\right)
\\
&
\qquad
+
\frac{\Theta^k_{y,i_1,n}}{\triangle x}
\left(
\begin{array}{c}
f_{12}(g_{0,i}^k)\\
f_{22}(g_{0,i}^k)
\end{array}
\right)
,\quad j=n,
\end{split}
\right.
\end{split}
}
\end{equation*}

\begin{equation*}
\begin{split}
&\frac{\partial}{\partial g_{0,i}}L(g_0^k,\rho^k,v^k;\psi^k,\Phi^k,\chi^k,\Theta^k)=\\
&
\quad
\begin{split}
&\sum_{j=2}^n\sum_{e_v}
\left(
\frac{\psi^k_{i-\frac{e_v}{2},j}-\psi^k_{i+\frac{e_v}{2},j}}{\triangle x}\right)\frac{1}{2}
{(v_{i+\frac{e_v}{2},j}^k)_v}^T
\left(
\begin{array}{cc}
f_{11}'(g_{0,i}^k)&f_{12}'(g_{0,i}^k)\\
f_{21}'(g_{0,i}^k)&f_{22}'(g_{0,i}^k)
\end{array}
\right)
(v_{i+\frac{e_v}{2},j}^k)_v
\\
&
+
\sum_{j=2}^{n-1}\left(
\frac{\psi^k_{i+\frac{e_v}{2},j}-\psi^k_{i+\frac{e_v}{2},j+1}}{\triangle x}
\right)
\left(
\begin{array}{cc}
f_{11}'(g_{0,i}^k)&f_{12}'(g_{0,i}^k)\\
f_{21}'(g_{0,i}^k)&f_{22}'(g_{0,i}^k)
\end{array}
\right)
(v_{i+\frac{e_v}{2},j}^k)_v
\\
&
+
\left(
\frac{({\psi^k_{i+\frac{e_v}{2},n}})_v}{\triangle t}
\right)^T
\left(
\begin{array}{cc}
f_{11}'(g_{0,i}^k)&f_{12}'(g_{0,i}^k)\\
f_{21}'(g_{0,i}^k)&f_{22}'(g_{0,i}^k)
\end{array}
\right)
(v_{i+\frac{e_v}{2},n}^k)_v
\\
&
-
\left(
\frac{({\psi^k_{i+\frac{e_v}{2},2}})_v}{\triangle t}
\right)^T
\left(
\begin{array}{cc}
f_{11}'(g_{0,i}^k)&f_{12}'(g_{0,i}^k)\\
f_{21}'(g_{0,i}^k)&f_{22}'(g_{0,i}^k)
\end{array}
\right)
(v_{i+\frac{e_v}{2},1}^k)_v
\\
&
+
\sum_{j=1}^n
\left(
\frac{\chi^k_{i-e_2,j}-\chi^k_{i,j}}{\triangle x}(f'_{11}(g_{0,i}^k),0)
+
\frac{\chi^k_{i-e_2,j}-\chi^k_{i,j}}{\triangle x}(0,f_{12}'(g_{0,i}^k))
-
\frac{\chi^k_{i-e_1,j}-\chi^k_{i,j}}{\triangle x}(f'_{21}(g_{0,i}^k),0)
\right.\\
&\left.
-\frac{\chi^k_{i-e_1,j}-\chi^k_{i,j}}{\triangle x}(0,f'_{22}(g_{0,i}^k))
\right)
(v_{i+\frac{e_v}{2},j}^k)_v
+
\sum_{j=1}^n
\left(
\frac{\Theta^k_{x,i_2,j}}{\triangle x}(f'_{11}(g_{0,i}^k),0)
+
\frac{\Theta^k_{x,i_2,j}}{\triangle x}(0,f'_{12}(g_{0,i}^k))
\right.
\\
&
\left.
+
\frac{\Theta^k_{y,i_1,j}}{\triangle x}(f'_{21}(g_{0,i}^k),0)
+
\frac{\Theta^k_{y,i_1,j}}{\triangle x}(0,f'_{22}(g_{0,i}^k))
\right)(v_{i+\frac{e_v}{2},j}^k)_v\\
&
+
\sum_{e_v}
\frac{\gamma}{\triangle t \triangle x^p}
|g_{0,i}^k-g_{0,i-e_v}^k|^{p-1}
sign(g_{0,i}^k-g_{0,i-e_v}^k)
\\
&
+
\sum_{e_v}
\frac{\gamma}{\triangle t \triangle x^p}
|g_{0,i}^k-g_{0,i+e_v}^k|^{p-1}
sign(g_{0,i}^k-g_{0,i+e_v}^k).
\end{split}
\end{split}
\end{equation*}

We apply a similar primal-dual algorithm to Model \ref{disc-inv-kernel}. Its details are omitted.

\begin{remark}
	\rm The algorithm is also parallelizable. In the primal step, we can also substitute the gradient descend with a proximal operator as in~\cite{chambolle_ergodic_2016} for acceleration, which keeps the algorithm parallelizable.
\end{remark}

\section{Complementary computational results}
\label{appen:comp}

\begin{test}
	\label{test:regular-1d-metric}
	\rm We applied our algorithm for Model \ref{disc-inv-ground-metric} with the 1-dimensional basis and without noise to the source data. We took $G_M=g_0$ and discretized the problem on a $50\times 30$ grid. We assumed to have the information about $G_M$ at a single pixel and its value was fixed in the iteration. The scaling parameter $\alpha,\alpha_0,\beta$ were taken as in (\ref{standard-para}), and $\gamma$ varied in $\{10^{-8},10^{-7},\ldots,10^{-3}\}$. Since we considered a smooth real ground metric, we took $p=2$. The iteration step size for all the primal variables was $\tau=2\times10^{-3}$, and the step size for the dual variables was $\sigma=10^{-3}$. We let our algorithm run $6\times10^4$ iterations and observed convergence. Then, we tested our algorithm on two ground metrics of different shapes. The results are depicted in Figure \ref{single_peak_result} and Figure \ref{double_peak_result}.
	\begin{figure}[htbp]
		\centering
		\includegraphics[width=0.9\textwidth]{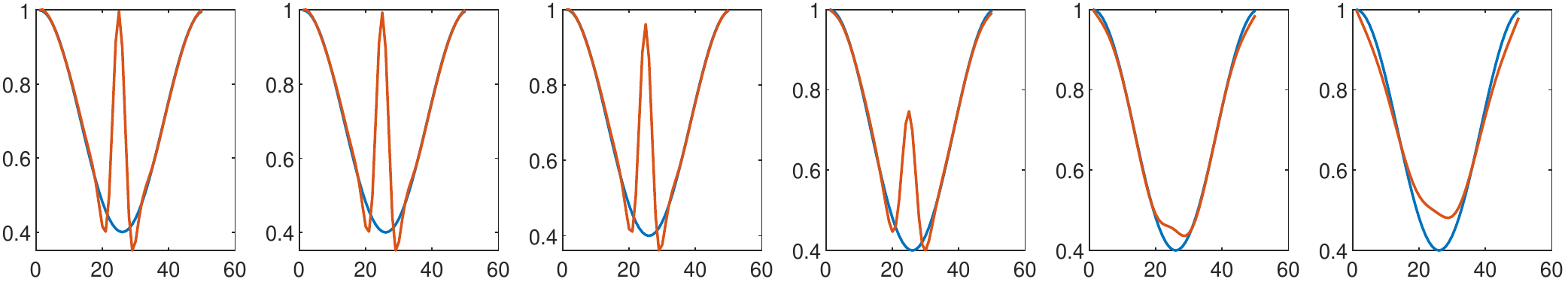}
		\caption{The result for $G_M(x)=1-0.6 \sin(\pi x)^2$ under the $H^2$ norm. From left to right,  $\gamma=10^{-8},10^{-7},\ldots,10^{-3}$. The red curve presents the leaned ground metric, and the blue curve depicts the real metric.}
		\label{single_peak_result}
	\end{figure}
	\begin{figure}[htbp]
		\centering
		\includegraphics[width=0.9\textwidth]{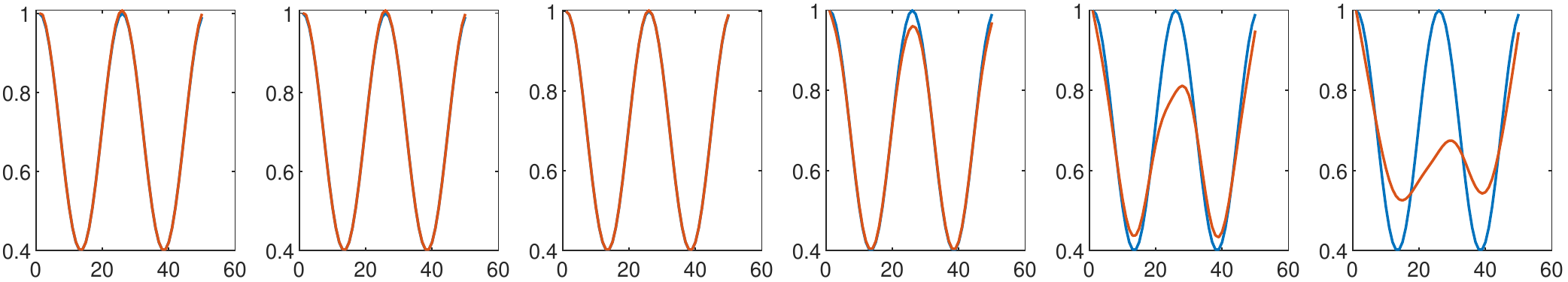}
		\caption{The result for $G_M(x)=1-0.6 \sin(2\pi x)^2$ under the $H^2$ norm. From left to right, $\gamma=10^{-8},10^{-7},\ldots,10^{-3}$.}
		\label{double_peak_result}
	\end{figure}
	\begin{remark}
		\rm From the numeric results, we conclude that the optimal parameter for the inverse model depends on the shape of the ground metric. When $g_0$ fluctuates more, we should choose a smaller $\gamma$ for the regularization. This matches our intuition.
	\end{remark}
\end{test}

\begin{test}
	\rm 
	We substituted the objective functions of (\ref{MFG-inv1}) and (\ref{MFG-inv2}) with (\ref{kinetic-KL}) and tested a similar primal-dual algorithm. In the ground metric inverse model, we took $p=2, \alpha=\alpha_0=0.01$, and $\gamma$ varied in $\{10^{-8},10^{-7},10^{-6},10^{-5}\}$. The step size for all primal variables was $\tau=2\times10^{-3}$, and the step size for dual variables was $\sigma=10^{-3}$. We iterated $3\times10^6$ times. The result can be found in \ref{JKO-metric}.
	
	In the kernel inverse model, we took $p=2,\alpha=\alpha_0=10$, $\gamma$ varied in $\{10^{-5},10^{-4},10^{-3},10^{-2}\}$. The iteration step size for all primal variables was $\tau=10^{-3}$, and the step size for dual variables was $\sigma=10^{-3}$ as well. We iterated $3\times10^6$ times. The result is in Figure \ref{JKO-kernel}.
	
	\begin{figure}[htbp]
		
		\centering
		\includegraphics[width=0.9\textwidth]{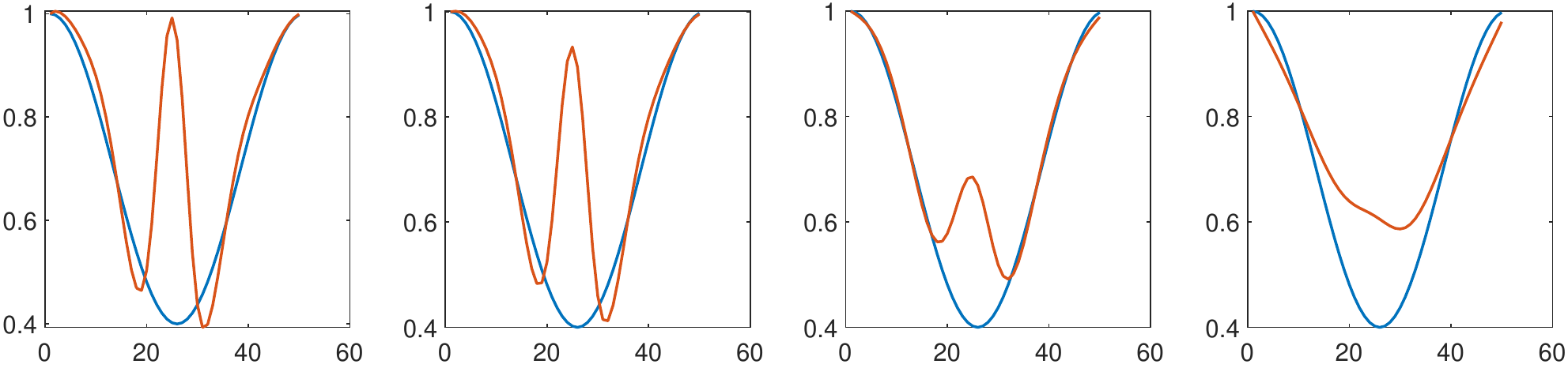}
		\caption{Recreated ground metric from the modified model with objective function (\ref{kinetic-KL}). From left to right, $\gamma=\{10^{-8},10^{-7},10^{-6},10^{-5}\}$.}
		\label{JKO-metric}
	\end{figure}
	
	\begin{figure}[htbp]
		
		\centering
		\includegraphics[width=0.9\textwidth]{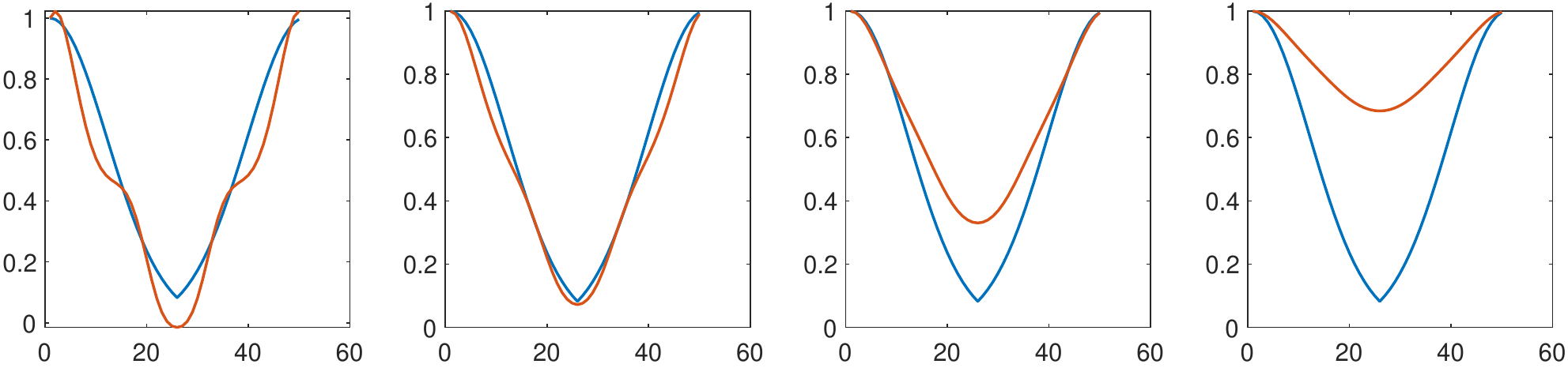}
		\caption{Recreated convolution kernel from the modified model with objective function (\ref{kinetic-KL}). From left to right, $\gamma=\{10^{-5},10^{-4},10^{-3},10^{-2}\}$.}
		\label{JKO-kernel}
	\end{figure}
\end{test}

\begin{test}
	\rm
	In this test, we used 2-dimensional noisy examples. We designed our experiment based on the data in Test \ref{test:2d-metric}, and the data was corrupted by additive noise defined in (\ref{noise}). For each noise level, the scaling parameters were selected as in Table \ref{table:2d-noise-metric}. Other settings followed Test \ref{test:2d-metric}. The result is depicted in Figure \ref{2d-noise-metric}.
	\begin{table}[htbp]
		\centering
		\begin{tabular}{c c c c c c}
			\hline
			Test & $\epsilon^*$ & $\alpha$ & $\alpha_0$& $\beta$ &$\gamma$ \\ [0.5ex]
			\hline\hline
			1 & 0.1 & $10/\|\hat{\rho}\|^2$ & 0 & $1/\hat{\vv}^T\hat{\vv}$ & 0.01\\
			2 & 0.4 & $1/\|\hat{\rho}\|^2$ & 0 & $1/\hat{\vv}^T\hat{\vv}$ & 0.01\\
			3 & 1 & $100/\|\hat{\rho}\|^2$ & 0 & $1/\hat{\vv}^T\hat{\vv}$ & 1 \\ [1ex]
			\hline
		\end{tabular}
		\caption{The scaling parameters adopted in each noisy test}
		\label{table:2d-noise-metric}
	\end{table}

	\begin{figure}[htbp]
		\centering
		\includegraphics[width=0.9\textwidth]{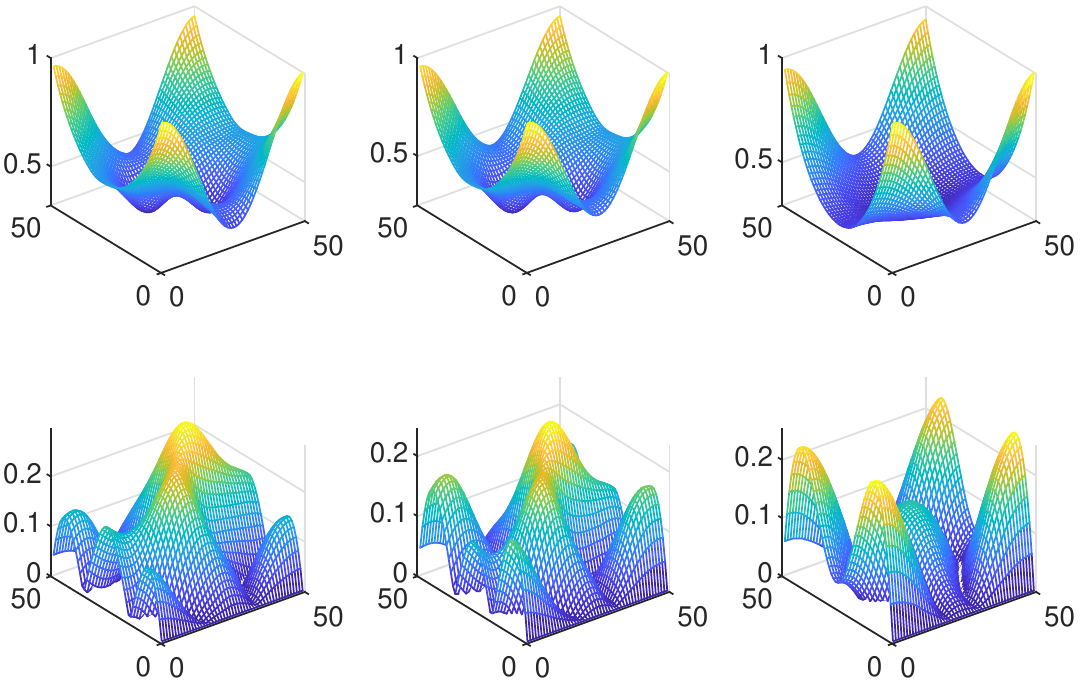}
		\caption{The recreated metric kernel from noisy data. In the first row, from right to left, the figures correspond to recreated metric kernel $g_0$ w.r.t. noise level $\epsilon^*=\{0.1,0.4,1\}$. The figures in the second row stand for the absolute difference between the learned $g_0$ and the real data $|g_0-\bar{g_0}|$ in each case.}
		\label{2d-noise-metric}
	\end{figure}
\end{test}

\begin{test}
	\rm In this test, we ran Algorithm \ref{alg:Bregman} and used the same data as Test \ref{test:kernel}, and the noise level was $\epsilon^*=1$. We set $\gamma=10^{-1}$, and the results of the first 11 Bregman iterations are depicted in Figure \ref{Breg_kernel_-1_1}.
	\begin{figure}[htbp]
		\centering
		\includegraphics[width=1\textwidth]{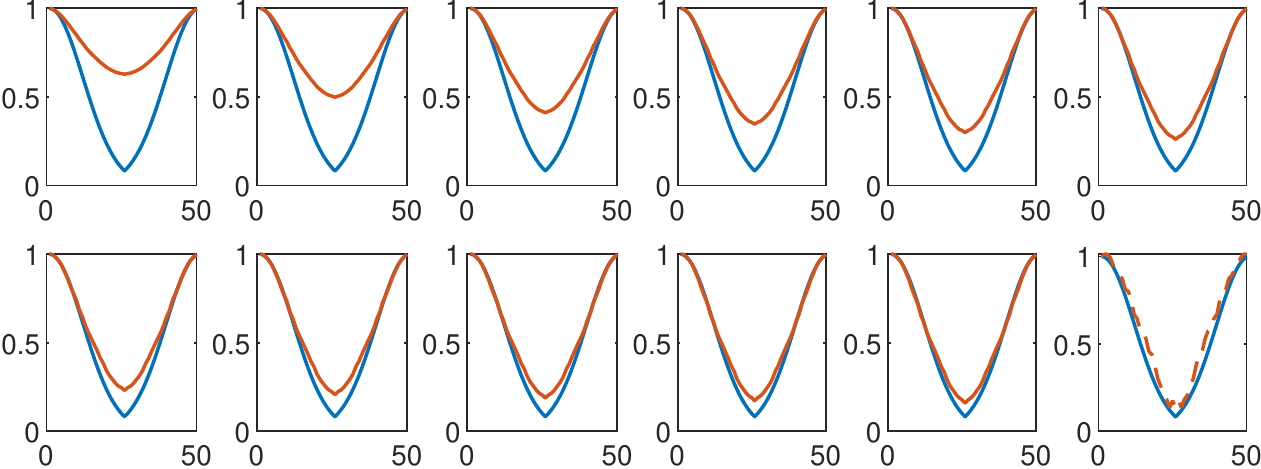}
		\caption{The recreated convolution kernels from noisy data from the first 11 Bregman iterations. $\gamma=10^{-1}$.
			(The last figure shows the recreated convolution kernel from Algorithm \ref{alg:primal-dual} with the nearly optimal $\gamma=10^{-3}$.)}
		\label{Breg_kernel_-1_1}
	\end{figure}
\end{test}

\end{document}